\documentclass[reqno]{amsart}

\usepackage{amsmath}
\usepackage{amsfonts}
\usepackage{amssymb,enumerate}
\usepackage{amsthm}
\usepackage[all]{xy}
\usepackage{rotating}
\usepackage{hyperref,color}

\theoremstyle{plain}
\newtheorem{lem}{Lemma}[section]
\newtheorem{cor}[lem]{Corollary}
\newtheorem{prop}[lem]{Proposition}
\newtheorem{thm}[lem]{Theorem}

\theoremstyle{definition}
\newtheorem{defn}[lem]{Definition}
\newtheorem{ex}[lem]{Example}
\newtheorem{question}[lem]{Question}

\newtheorem{disc}[lem]{Remark}

\newtheorem{fact}[lem]{Fact}

\newtheorem{notation}[lem]{Notation}

\newtheorem{subprops}{}[lem]

\newcommand{\cat}[1]{\mathcal{#1}}

\newcommand{\catd}{\cat{D}}

\newcommand{\cata}{\cat{A}}

\newcommand{\catb}{\cat{B}}

\newcommand{\catac}{\cat{A}_C}

\newcommand{\catbc}{\cat{B}_C}

\newcommand{\pd}{\operatorname{pd}}

\newcommand{\id}{\operatorname{id}}	
\newcommand{\fd}{\operatorname{fd}}

\newcommand{\depth}{\operatorname{depth}}

\newcommand{\mspec}{\operatorname{m-Spec}}

\newcommand{\HH}{\operatorname{H}}
\newcommand{\Hom}{\operatorname{Hom}}	

\newcommand{\spec}{\operatorname{Spec}}

\newcommand{\shift}{\mathsf{\Sigma}}

\newcommand{\ideal}[1]{\mathfrak{#1}}
\newcommand{\m}{\ideal{m}}
\newcommand{\n}{\ideal{n}}
\newcommand{\p}{\ideal{p}}
\newcommand{\q}{\ideal{q}}

\newcommand{\fa}{\ideal{a}}
\newcommand{\fb}{\ideal{b}}

\newcommand{\comp}[1]{\widehat{#1}}

\newcommand{\supp}{\operatorname{supp}}
\newcommand{\Supp}{\operatorname{Supp}}

\newcommand{\VE}{\operatorname{V}}
\newcommand{\cosupp}{\operatorname{co-supp}}

\newcommand{\bbz}{\mathbb{Z}}
\newcommand{\bbn}{\mathbb{N}}

\newcommand{\xra}{\xrightarrow}

\newcommand{\vf}{\varphi}

\newcommand{\x}{\underline{x}}

\renewcommand{\geq}{\geqslant}
\renewcommand{\leq}{\leqslant}

\newcommand{\Ext}[4][R]{\operatorname{Ext}_{#1}^{#2}(#3,#4)}	
\newcommand{\Rhom}[3][R]{\mathbf{R}\!\operatorname{Hom}_{#1}(#2,#3)}	
\newcommand{\Lotimes}[3][R]{#2\otimes^{\mathbf{L}}_{#1}#3}
\newcommand{\Otimes}[3][R]{#2\otimes_{#1}#3}
\renewcommand{\Hom}[3][R]{\operatorname{Hom}_{#1}(#2,#3)}	
\newcommand{\Tor}[4][R]{\operatorname{Tor}^{#1}_{#2}(#3,#4)}

\newcommand{\LL}[2]{\mathbf{L}\Lambda^{\ideal{#1}}(#2)}

\newcommand{\RG}[2]{\mathbf{R}\Gamma_{\ideal{#1}}(#2)}

\newcommand{\Comp}[2]{\widehat{#1}^{\ideal{#2}}}

\newcommand{\ssm}{\smallsetminus}

\newcommand{\catdfb}{\catd_{\text{b}}^{\text{f}}}
\newcommand{\catdb}{\catd_{\text{b}}}
\newcommand{\catdf}{\catd^{\text{f}}}

\newcommand{\LLa}[2]{\mathbf{L}\Lambda^{\mathfrak{#1}\Comp R{#1}}(#2)}
\newcommand{\LLno}[1]{\mathbf{L}\Lambda^{\ideal{#1}}}

\newcommand{\RGno}[1]{\mathbf{R}\Gamma_{\ideal{#1}}}

\newcommand{\fromRG}[2]{\varepsilon_{\ideal{#1}}^{#2}}
\newcommand{\fromRGno}[1]{\varepsilon_{\ideal{#1}}}

\newcommand{\toLLno}[1]{\vartheta^{\ideal{#1}}}

\newcommand{\PRhom}[3][R]{\mathbf{R}\!\operatorname{Hom}_{#1}\left(#2,#3\right)}	
\newcommand{\LLS}[2]{\mathbf{L}\Lambda^{\ideal{#1} S}(#2)}

\numberwithin{equation}{lem}

\begin{document}

\bibliographystyle{amsplain}

\author{Sean Sather-Wagstaff}

\address{Department of Mathematical Sciences,
Clemson University,
O-110 Martin Hall, Box 340975, Clemson, S.C. 29634
USA}

\email{ssather@clemson.edu}

\urladdr{https://ssather.people.clemson.edu/}

\thanks{
Sean Sather-Wagstaff was supported in part by a grant from the NSA}

\author{Richard Wicklein}

\address{Richard Wicklein, Mathematics and Physics Department, MacMurray College, 447 East College Ave., Jacksonville, IL 62650, USA}

\email{richard.wicklein@mac.edu}

\title{Adic Foxby Classes}




\begin{abstract}
We continue our work on adic semidualizing complexes over a commutative noetherian ring $R$ 
by investigating the associated Auslander and Bass classes (collectively known as Foxby classes), following Foxby and Christensen. 
Fundamental properties of these classes include Foxby Equivalence, which provides an equivalence between the Auslander and Bass classes associated to a given adic semidualizing complex. 
We prove a variety of stability results for these classes, for instance, with respect to $\Lotimes F-$ where $F$ is an $R$-complex finite flat dimension, including special converses of these results. 
We also investigate change of rings and local-global properties of these classes.
\end{abstract}

\maketitle

\tableofcontents

\section{Introduction} \label{sec130805a}
Throughout this paper let $R$ be acommutative noetherian ring, let $\fa \subsetneq R$ be a proper ideal of $R$, and let $\Comp{R}{a}$ be the $\fa$-adic completion of $R$.
Let $K=K^R(\x)$ denote the Koszul complex over $R$ on a  generating sequence $\x=x_1,\ldots,x_n$ for $\fa$.
We work in the derived category $\catd(R)$ with objects the $R$-complexes
indexed homologically
$X=\cdots\to X_i\to X_{i-1}\to\cdots$;
see, e.g., \cite{hartshorne:rad,verdier:cd,verdier:1} for  foundations of this construction and Section~\ref{sec140109b} for background and notation.
Isomorphisms in $\catd(R)$ are identified by the symbol $\simeq$.
We also consider the  full triangulated subcategory
$\catdb(R)$, with objects the $R$-complexes $X$ such that $\HH_i(X)=0$ for $|i|\gg 0$.
The appropriate derived functors of Hom and $\otimes$ are denoted $\mathbf{R}\!\operatorname{Hom}$ and $\otimes^{\mathbf{L}}$.

\

This paper is part 6 of a series of papers on homological constructions over $R$; see also~\cite{sather:afbha,sather:afcc,sather:asc,sather:elclh,sather:scc}.
The genesis of the current paper goes back at least to Auslander and Bridger's monograph~\cite{auslander:smt} on G-dimension of 
finitely generated modules, defined in terms of resolutions by modules of G-dimension 0, i.e., totally reflexive modules. 
This was extended to the non-finite arena by Enochs, Jenda, and Torrecillas~\cite{enochs:gipm,enochs:gf} yielding the 
G-projective, G-flat, and G-injective dimensions.

One difficulty with these dimensions is found in their definitions in terms of resolutions. 
As opposed to the standard characterization of projective dimension in terms of Ext-vanishing, for instance, a functorial characterization of the modules of finite G-projective dimension 
took significantly more work. 
This goal was achieved, first for Cohen-Macaulay rings admitting a dualizing (i.e., canonical) module,
by Enochs, Jenda, and Xu~\cite{enochs:fdgipm},
then for rings admitting a dualizing complex, by Christensen, Frankild, and Holm~\cite{christensen:ogpifd}, 
then in general by Esmkhani and Tousi~\cite{esmkhani:ghdac} and Christensen and Sather-Wagstaff~\cite{christensen:tgdrh}.
The first of these uses Foxby's Auslander and Bass classes~\cite{foxby:gmarm} with respect to a dualizing module, and the later ones use 
Avramov and Foxby's~\cite{avramov:rhafgd} more general Auslander and Bass classes with respect to a dualizing complex.

Motivated in part by Avramov and Foxby's~\cite{avramov:rhafgd} use of \emph{relative} dualizing complexes to study ring homomorphisms of finite G-dimension,
Christensen~\cite{christensen:scatac} introduced and studied Auslander and Bass classes with respect to semidualizing complexes.
A complex $C\in\catdb(R)$ with finitely generated homology is \emph{semidualizing} if the natural homothety morphism $R\to\Rhom CC$ is an isomorphism
in $\catd(R)$. Examples of these include dualizing complexes (in particular, dualizing modules), relative dualizing complexes, and the free $R$-module of rank 1, each of which has
important duality properties.
When $C$ is an $R$-\emph{module}, 
it is semidualizing over $R$ if it is finitely generated with $\Ext iCC=0$ for all $i\geq 1$ such that the natural homothety map $R\to\Hom CC$ is an isomorphism

The \emph{Bass class} $\catbc(R)$ with respect to $C$ is the class of $R$-complexes $X\in\catdb(R)$ with $\Rhom CX\in\catdb(R)$ where 
the natural morphism $\Lotimes{C}{\Rhom{C}{X}}\to X$ is  an isomorphism in $\catd(R)$.
The Auslander class $\catac(R)$ is defined similarly; see~\ref{defn121104a}.
(Fact~\ref{rmk140218b} describes these classes in the case of modules.)
By work of Holm and J\o rgensen~\cite{holm:smarghd}, these classes give functorial characterizations of certain generalized Gorenstein homological dimensions.
In summary, these classes are powerful tools for studying homological dimensions of modules and complexes.

Note that the definitions of $\catac(R)$ and $\catbc(R)$ do not \emph{a priori} require $C$ to be semidualizing.
However, when $C$ is not semidualizing, these classes tend to lose their nice properties, especially when $C$ has non-finitely generated homology. 
On the other hand, the class of semidualizing $R$-modules misses some  modules and complexes  with important duality properties, e.g., the injective hull
$E_R(R/\m)$ over a local ring $(R,\m)$, used for Matlis duality and Grothendieck's local duality. 
To fill this gap, in~\cite{sather:asc} we introduce and study the $\fa$-adic semidualizing complexes; see Definition~\ref{def120925e} below. 
In the case $\fa=0$, these are exactly Christensen's semidualizing complexes.
When $(R,\m)$ is local and $\fa=\m$, these recover Kubik's~\cite{kubik:qdm} quasi-dualizing modules (e.g., $E_R(R/\m)$) as a special case. 

The point of the current paper is to investigate the Auslander and Bass classes with respect to an $\fa$-adic semidualizing complex $M$. 
Even though $M$ does not generally have finitely generated homology, the definition allows us to retain many of the nice properties from Christensen's setting. 
For instance, we have the following version of Foxby Equivalence (originally from~\cite{avramov:rhafgd,christensen:scatac,foxby:gmarm});
it is Theorem~\ref{thm121116a} below, one of several foundational properties documented in Section~\ref{sec130818d}.

\begin{thm}\label{thm121116ax}
Let $M$ be an $\mathfrak{a}$-adic semidualizing $R$-complex. 
\begin{enumerate}[\rm(a)]
\item\label{thm121116ax1}
The functors $\Rhom{M}{-}\colon \mathcal{B}_M(R) \rightarrow \mathcal{A}_M(R)$ and $\Lotimes{M}{-}\colon \mathcal{A}_M(R) \rightarrow \mathcal{B}_M(R)$ are quasi-inverse equivalences. 
\item\label{thm121116ax2}
An $R$-complex $Y\in\catd(R)$ is in $\mathcal{B}_{M}(R)$ if and only if $\Rhom{M}{Y} \in  \mathcal{A}_M(R)$ and $\supp_{R}(Y) \subseteq \VE(\fa)$.
\item\label{thm121116ax3}
An $R$-complex  $X\in\catd(R)$ is in $\mathcal{A}_M(R)$ if and only if one has $\Lotimes{M}{X} \in  \mathcal{B}_M(R)$ and $\operatorname{co-supp}_{R}(X) \subseteq \VE(\fa)$.
\end{enumerate}
\end{thm}

A difference between this result and its predecessors is the presence of support/co-support conditions; see Definition~\ref{defn130503a}.
On the other hand, these conditions are present in these earlier results, but they are invisible. Indeed, we have $\fa=0$ in those cases, so the condition
$\supp_{R}(Y) \subseteq \VE(0)=\spec(R)$ is satisfied trivially, and similarly for $\cosupp_R(X)$. 
This is a necessary feature of our constructions.
Another difference worth noticing is the lack of boundedness assumptions in parts~\eqref{thm121116ax2}--\eqref{thm121116ax3}.

Section~\ref{sec130818dz} is devoted to stability properties of these classes, i.e., their behavior with
respect to direct sum and product, in addition to well-behaved derived functors. 
For example, the next result is Theorem~\ref{prop151123a} from the body of the paper.

\begin{thm}\label{prop151123ar}
Let $M$ be an $\fa$-adic semidualizing $R$-complex.
Let $F\in\catdb(R)$ be such that $\fd_R(F)<\infty$, and let $X\in\catd(R)$.
If $X\in\catb_M(R)$, then $\Lotimes XF\in\catb_M(R)$ and $\supp_R(X)\subseteq\VE(\fa)$.
The converse of this statement holds when 
at least one of the following conditions is satisfied.
\begin{enumerate}[\rm(1)]
\item\label{prop151123ar1}
$F$ is $\fa$-adically finite such that $\supp_R(F)=\VE(\fa)$.
\item\label{prop151123ar2}
There is a  homomorphism $\vf\colon R\to S$ of commutative noetherian rings
with $\fa S\neq S$ such that $F\in\catdb(S)$ is $\fa S$-adically finite over $S$ with $\vf^*(\supp_S(F))\supseteq\VE(\fa)\bigcap\mspec(R)$,
and we have $\Lotimes KX\in\catdf(R)$.
Here
$\vf^*\colon\spec(S)\to\spec(R)$ is the induced map.
\item\label{prop151123ar3}
$F$ is a flat $R$-module with $\supp_R(F)\supseteq\VE(\fa)\bigcap\mspec(R)$, e.g., $F$ is faithfully flat, e.g., free.
\end{enumerate}
\end{thm}

In Section~\ref{sec150527a} we focus on transfer properties for these classes with respect to a ring homomorphism $\vf\colon R\to S$.
As a sample, here is Theorem~\ref{prop151127a} from this section.

\begin{thm}\label{prop151127ar}
Let $M$ be an $\fa$-adic semidualizing $R$-complex.
Let $\vf\colon R\to S$ be a homomorphism of commutative noetherian rings, and
assume that $\fd_R(S)<\infty$.
Let $X\in\catd(R)$ be given, and consider the following conditions.
\begin{enumerate}[\rm(i)]
\item \label{prop151127ar1}
$X\in\catb_M(R)$.
\item \label{prop151127ar2}
$\Lotimes SX\in\catb_{M}(R)$ and $\supp_R(X)\subseteq\VE(\fa)$.
\item \label{prop151127ar3}
$\Lotimes SX\in\catb_{\Lotimes SM}(S)$ and $\supp_R(X)\subseteq\VE(\fa)$.
\end{enumerate}
Then we have~\eqref{prop151127a1}$\implies$\eqref{prop151127a2}$\iff$\eqref{prop151127a3}.
The conditions~\eqref{prop151127a1}--\eqref{prop151127a3} are equivalent when 
at least one of the following conditions is satisfied.
\begin{enumerate}[\rm(1)]
\item\label{prop151127ar4}
$S$ is $\fa$-adically finite over $R$ such that $\supp_R(S)=\VE(\fa)$.
\item\label{prop151127ar5}
$\supp_R(S)\supseteq\VE(\fa)\bigcap\mspec(R)$,
and $\Lotimes KX\in\catdf(R)$.
\item\label{prop151127ar6}
$S$ is  flat over $R$ with $\supp_R(S)\supseteq\VE(\fa)\bigcap\mspec(R)$, e.g., $S$ is faithfully flat.
\end{enumerate}
\end{thm}

This result yields local-global behavior for Bass classes; see Theorem~\ref{thm151214a}.
This section also contains versions of these results for Auslander classes.

A reader familiar with the paper of Christensen~\cite{christensen:scatac} will undoubtedly see numerous similarities between that paper and this one.
However, the fact that $\fa$-adic semidualizing complexes do not usually have finitely generated homology makes for some technical and subtle differences.
On the other hand, some of our results, including parts of Theorem~\ref{prop151127ar}, are new even for Christensen's setting. 

\section{Background}\label{sec140109b} 

\subsection*{Derived Categories}
We consider the following full  subcategories of $\catd(R)$.

\

$\catd_+(R)$: objects are the complexes $X$ with $\HH_i(X)=0$ for $i\ll 0$.

$\catd_-(R)$: objects are the complexes $X$ with $\HH_i(X)=0$ for $i\gg 0$. 

$\catdf(R)$: objects are the complexes $X$ with $\HH_i(X)$ finitely generated for all $i$.

\

\noindent
Intersections of these categories are designated with two ornaments, e.g., $\catdfb(R)=\catdb(R)\bigcap\catdf(R)$.
The $i$th shift (or suspension) of an $R$-complex $X$ is denoted $\shift^iX$, and
the
\emph{supremum} and \emph{infimum} of  $X$ are
\begin{align*}
\sup(X)&=\sup\{ i\in\bbz\mid\HH_i(X)\neq 0\}\\
\inf(X)&=\inf\{ i\in\bbz\mid\HH_i(X)\neq 0\}
\end{align*}
with the conventions $\sup\emptyset=-\infty$ and $\inf\emptyset=\infty$.

\begin{fact}\label{disc151112a}
Let $Y,Z\in\catd(R)$.
\begin{enumerate}[(a)]
\item\label{disc151112a1}
By definition, one has $\sup(Z)<\infty$ if and only if $Z\in\catd_-(R)$,
and one has $\inf(Z)>-\infty$ if and only if $Z\in\catd_+(R)$.
\item\label{disc151112a5}
By~\cite[Lemma~2.1(1)]{foxby:ibcahtm},
there is an inequality
\begin{gather*}
\sup(\Rhom YZ)\leq\sup(Z)-\inf(Y).
\end{gather*}
\end{enumerate}
\end{fact}

The next lemma is routine, but we include a proof for the sake of completeness.

\begin{lem}\label{prop151123z}
If $X,Y\in\catd_+(R)$ and $Z\in\catd_-(R)$ are given such that the complexes $\Lotimes KX$, ${\Lotimes KY}$, and ${\Lotimes KZ}$ are in $\catdf(R)$,
then one has $\Lotimes K{(\Lotimes XY)}\in\catdf_+(R)$ and $\Lotimes K{\Rhom XZ}\in\catdf_-(R)$.
\end{lem}

\begin{proof}
First, we observe that our assumptions on $X$ and $Y$ yield the following.
$$\Lotimes K{(\Lotimes K{(\Lotimes XY)})}\simeq\Lotimes{(\Lotimes KX)}{(\Lotimes KY)}\in\catdf_+(R)$$
Let $n$ be the length of our given generating sequence for $\fa$.
Since each module $\HH_i(\Lotimes K{(\Lotimes XY)})$ is annihilated by $\fa$, we have
$$\HH_i(\Lotimes K{(\Lotimes K{(\Lotimes XY)})})\cong\bigoplus_{j=0}^n\HH_{i+j}(\Lotimes K{(\Lotimes XY)})^{\binom{n}{j}}$$
In particular, $\HH_i(\Lotimes K{(\Lotimes XY)})$ is a summand of the finitely generated module $\HH_i(\Lotimes K{(\Lotimes K{(\Lotimes XY)})})$,
so it is finitely generated as well. This shows that we have $\Lotimes K{(\Lotimes XY)}\in\catdf(R)$.
The conclusion
$\Lotimes K{\Rhom XZ}\in\catdf(R)$ is obtained similarly, using the isomorphisms
$$\Lotimes K{(\Lotimes K{(\Rhom XZ)})}\simeq\Rhom{\Rhom KX}{\Lotimes KZ}$$
and $\Lotimes KX\simeq\shift^n\Rhom KX$.
\end{proof}

\subsection*{Homological Dimensions}
An $R$-complex $F$ is \emph{semi-flat}\footnote{In the literature, semi-flat complexes are sometimes called ``K-flat'' or ``DG-flat''.} 
if it consists of flat $R$-modules and the functor $\Otimes F-$ respects quasiisomorphisms,
that is, if it respects injective quasiisomorphisms (see~\cite[1.2.F]{avramov:hdouc}).
A \emph{semi-flat resolution} of an $R$-complex $X$ is a quasiisomorphism $F\xra\simeq X$ such that $F$ is semi-flat.
An $R$-complex $X$ has \emph{finite flat dimension} if it has a bounded semi-flat resolution;
specifically, we have
$$\fd_R(X)=\inf\{\sup\{i\mid F_i\neq 0\}\mid\text{$F\xra\simeq X$ is a semi-flat resolution}\}.$$
The  projective and injective versions of these notions are defined similarly. 

For the following items, consult~\cite[Section 1]{avramov:hdouc} or~\cite[Chapters 3 and 5]{avramov:dgha}.
Bounded below  complexes of flat modules are semi-flat, 
bounded below  complexes of projective modules are semi-projective, and
bounded above  complexes of injective modules are semi-injective.
Semi-projective $R$-complexes are semi-flat. 
Every $R$-complex admits a semi-projective (hence, semi-flat) resolution  and a semi-injective resolution.

\subsection*{Derived Local (Co)homology}
The next notions go back to Grothendieck~\cite{hartshorne:lc} and Matlis~\cite{matlis:kcd,matlis:hps};
see also~\cite{lipman:lhcs,greenlees:dfclh,lipman:llcd,yekutieli:hct}. 
Let $\Lambda^{\fa}$ denote the $\fa$-adic completion functor, and let
$\Gamma_{\fa}$ be the $\fa$-torsion functor, i.e.,
for an $R$-module $M$ we have
$$\Lambda^{\fa}(M)=\Comp Ma
\qquad
\qquad
\qquad
\Gamma_{\fa}(M)=\{ x \in M \mid \fa^{n}x=0 \text{ for } n \gg 0\}.$$ 
A module $M$ is \textit{$\fa$-torsion} if $\Gamma_{\fa}(M)=M$.

The associated left and right derived functors (i.e., \emph{derived local homology and cohomology} functors)
are  $\LL a-$ and $\RG a-$.
Specifically, given an $R$-complex $X\in\catd(R)$ and a semi-flat resolution $F\xra\simeq X$ and a 
semi-injective resolution $X\xra\simeq I$, then we have $\LL aX\simeq\Lambda^{\fa}(F)$ and $\RG aX\simeq\Gamma_{\fa}(I)$.

\begin{fact}\label{fact130619b}
By~\cite[Theorem~(0.3) and Corollary~(3.2.5.i)]{lipman:lhcs}, there are natural isomorphisms
of functors
\begin{align*}
\RG a-\simeq\Lotimes{\RG aR}{-}&&
\LL a-\simeq\Rhom{\RG aR}{-}.
\end{align*}
Note that we have $\pd_R(\RG aR)<\infty$, via the telescope complex of~\cite{greenlees:dfclh}.
Thus, if $F\in\catdb(R)$ has finite flat dimension, then so has
$\LL aF\simeq\Rhom{\RG aR}F$.
\end{fact}

\subsection*{Support and Co-support}
The following notions of support and co-support, also crucial for our work, are due to Foxby~\cite{foxby:bcfm} and Benson, Iyengar, and Krause~\cite{benson:csc}.

\begin{defn}\label{defn130503a}
Let $X\in\catd(R)$.
The \emph{small support} and  \emph{small co-support} of $X$ are
\begin{align*}
\operatorname{supp}_R(X)
&=\{\mathfrak{p} \in \operatorname{Spec}(R)\mid \Lotimes{\kappa(\p)}X\not\simeq 0 \} \\
\cosupp_{R}(X)
&=\{\mathfrak{p} \in \operatorname{Spec}(R)\mid \Rhom{\kappa(\p)}X\not\simeq 0 \} 
\end{align*}
where $\kappa(\p):=R_\p/\p R_\p$.
\end{defn}

Much of the following is from~\cite{foxby:bcfm} when $X$ and $Y$ are appropriately bounded and from~\cite{benson:lcstc,benson:csc} in general. 
We refer to~\cite{sather:scc} as a matter of convenience.

\begin{fact}\label{cor130528a}
Let $X,Y\in\catd(R)$. 

\begin{subprops}
\label{cor130528a1}
We have $\supp_R(X)=\emptyset$ if and only if $X\simeq 0$ if and only if $\cosupp_R(X)=\emptyset$,
because of~\cite[Fact~3.4 and Proposition~4.7(a)]{sather:scc}.
\end{subprops}

\begin{subprops}
\label{cor130528a2}
By~\cite[Propositions~3.12, 3.13, 4.10, and~4.11]{sather:scc} we have 
\begin{align*}
\supp_{R}(\Lotimes{X}{Y}) 
&= \supp_R(X)\bigcap\supp_R(Y)\\
\cosupp_{R}(\Rhom{X}{Y}) 
&= \supp_R(X)\bigcap\cosupp_R(Y)\\
\supp_{R}(\RG aY) 
&= \VE(\fa)\bigcap\supp_R(Y)\\
\cosupp_{R}(\LL a{Y}) 
&= \VE(\fa)\bigcap\cosupp_R(Y).
\end{align*}
\end{subprops}

\begin{subprops}
\label{cor130528a3}
We know that $\supp_R(X)\subseteq\VE(\fa)$ if and only if 
the natural morphism $\fromRG aX\colon\RG aX\to X$ is an isomorphism,
that is, if and only if each homology module $\HH_i(X)$ is $\fa$-torsion,
by~\cite[Proposition~5.4]{sather:scc}
and~\cite[Corollary~4.32]{yekutieli:hct}.
Dually, we have $\cosupp_R(X)\subseteq\VE(\fa)$ if and only if the natural morphism $X\to\LL aX$ in $\catd(R)$ is an isomorphism, 
by~\cite[Proposition~5.9]{sather:scc}.
Also, each homology module $\HH_i(X)$ is $\fa$-adically complete if and only if each $\HH_i(X)$ is $\fa$-adically separated
and $\cosupp_R(X)\subseteq\VE(\fa)$, by~\cite[Theorem~3]{yekutieli:sccmc}.
Since $\fa$ annihilates the homology of $\Lotimes KX$,
it follows from this  that $\supp(\Lotimes KX),\cosupp(\Lotimes KX)\subseteq\VE(\fa)$.
\end{subprops}
\end{fact}

The next three facts demonstrate some of the flexibility afforded by support and co-support conditions.

\begin{fact}[\protect{\cite[Lemmas~3.1(b) and 3.2(b)]{sather:afcc}}]
\label{lem150604a}
Let  $Z\in\catd(R)$.
If $\supp_R(Z)\subseteq\VE(\fa)$ or $\cosupp_R(Z)\subseteq\VE(\fa)$,
then 
one has $\Lotimes KZ\in\catdb(R)$ if and only if $Z\in\catdb(R)$.
\end{fact}

\begin{fact}[\protect{\cite[Lemma~2.8]{sather:elclh}}]
\label{fact151222c}
Let $X\in\catd(R)$ be such that  $\supp_R(X)\subseteq\VE(\fa)$.
Then the following natural transformations from Fact~\ref{cor130528a3}
are isomorphisms.
\begin{gather*}
\Lotimes X{-}\xra[\simeq]{\Lotimes X{\toLLno a}}\Lotimes X{\LL a-}
\\
\Rhom X{\RG a-} \xra[\simeq]{\Rhom X{\fromRGno a}}\Rhom{X}{-}
\end{gather*}
\end{fact}

\begin{fact}[\protect{\cite[Corollaries~5.3 and~5.8]{sather:scc}}]
\label{thm130318aqqc}
Let 
$f\colon Y \to Z$ be morphism in $\catd(R)$.
Assume that either  $\supp_R(Y), \supp_R(Z) \subseteq\VE(\fa)$ or  $\cosupp_R(Y), \cosupp_R(Z) \subseteq\VE(\fa)$.
Then  
$f$ is an isomorphism in $\catd(R)$ if and only if 
$\Lotimes{K}{f}$ is an isomorphism in $\catd(R)$.
Here is the main idea of the proof, for instance, in the case $\supp_R(Y), \supp_R(Z) \subseteq\VE(\fa)$. 
Consider an exact triangle $Y\xra fZ\to A\to $ in $\catd(R)$. 
The support assumptions on $Y$ and $Z$ imply that we have $\supp_R(A)\subseteq\VE(\fa)$.
Since $\Lotimes{K}{f}$ is an isomorphism, when we tensor this triangle with $K$, we find that $\Lotimes KA\simeq 0$.
On the other hand, we have $\supp_R(K)=\VE(\fa)\supseteq\supp_R(A)$, so Fact~\ref{cor130528a2} implies that $\supp_R(A)=\emptyset$ and so $A\simeq 0$ by Fact~\ref{cor130528a1}.
Thus, our triangle shows that $f$ is an isomorphism. 

The point of discussing this proof explicitly is as follows.
We have many results in~\cite[Sections~3--5]{sather:afbha} of the following form: given a functor $F$ and an
$R$-complex $A$ such that $F(A)\simeq 0$, nice assumptions on $F$ and $A$ imply that we have $A\simeq 0$. 
Using the logic of the previous paragraph, we conclude that if $f$ is a morphism between nice complexes
such that $F(f)$ is an isomorphism, then $f$ is an isomorphism. 
To keep things reasonable, we do not state every possible variation on this theme, though we use this idea several times below. 
\end{fact}

\subsection*{Adic Finiteness}
The following fact and definition take their cues from work of 
Hartshorne~\cite{hartshorne:adc},
Kawasaki~\cite{kawasaki:ccma,kawasaki:ccc}, and
Melkersson~\cite{melkersson:mci}.

\begin{fact}[\protect{\cite[Theorem~1.3]{sather:scc}}]
\label{thm130612a}
For $X\in\catd_{\text b}(R)$, the next conditions are equivalent.
\begin{enumerate}[\rm(i)]
\item\label{cor130612a1}
One has $\Lotimes{K(\underline{y})}{X}\in\catdfb(R)$  for some (equivalently for every) generating sequence $\underline{y}$ of $\fa$.
\item\label{cor130612a2}
One has  $\Lotimes{X}{R/\mathfrak{a}}\in\catd^{\text{f}}(R)$.
\item\label{cor130612a3}
One has  $\Rhom{R/\mathfrak{a}}{X}\in\catd^{\text{f}}(R)$.
\end{enumerate}
\end{fact}

\begin{defn}\label{def120925d}
An $R$-complex $X\in\catdb(R)$ is \emph{$\mathfrak{a}$-adically finite} if it satisfies the equivalent conditions of Fact~\ref{thm130612a} and 
$\operatorname{supp}_R(X) \subseteq \operatorname{V}(\mathfrak{a})$; see Fact~\ref{cor130528a3}.
\end{defn}

\begin{ex}\label{ex160206a}
Let $X\in\catdb(R)$ be given.
\begin{enumerate}[(a)]
\item \label{ex160206a1}
If $X\in\catdfb(R)$, then we have $\supp_R(X)=\VE(\fb)$ for some ideal $\fb$, and it follows that $X$ is $\fa$-adically finite
whenever $\fa\subseteq\fb$. (The case $\fa=0$ is from~\cite[Proposition~7.8(a)]{sather:scc}, and the general case follows readily.)
\item \label{ex160206a2}
$K$ and $\RG aR$ are $\fa$-adically finite, by~\cite[Fact~3.4 and Theorem~7.10]{sather:scc}.
\item \label{ex160206a3}
The homology modules of $X$ are artinian if and only if there is an ideal $\fa$ of finite colength (i.e., such that $R/\fa$ is artinian)
such that $X$ is $\fa$-adically finite, by~\cite[Proposition~5.11]{sather:afcc}.
\end{enumerate}
\end{ex}

As the name suggests, adically finite complexes can behave like the complexes in $\catdfb(R)$. 
The next fact gives a taste of this. See~\cite{sather:afbha,sather:afcc} for other such results.

\begin{fact}[\protect{\cite[Theorems~4.1(b) and~4.2(b)]{sather:afcc}}]\label{fact151222a}
Let $M$ be an $\fa$-adically finite $R$-complex, and let $Y,V,Z\in\catd(R)$ be such that $\supp_R(V)\subseteq\VE(\fa)$,
e.g., $V=M$ or $V=K$.
Consider the natural tensor evaluation and Hom-evaluation morphisms
\begin{gather*}
\Lotimes{\Rhom MY}{Z}\xra{\omega_{MYZ}}\Rhom M{\Lotimes YZ}\\
\Lotimes M{\Rhom YZ}\xra{\theta_{MYZ}}\Rhom {\Rhom MY}Z.
\end{gather*}
\begin{enumerate}[\rm(a)]
\item \label{fact151222a1}
If $Y\in\catd_-(R)$ and  either $\fd_R(M)$ or $\fd_R(Z)$ is finite, then
the induced morphisms $\Lotimes V{\omega_{MYZ}}$ and $\Rhom V{\omega_{MYZ}}$  are isomorphisms.
\item \label{fact151222a2}
If $Y\in\catdb(R)$ and  either $\fd_R(M)$ or $\id_R(Z)$ is finite, then
the induced morphisms $\Lotimes V{\theta_{MYZ}}$ and $\Rhom V{\theta_{MYZ}}$  are isomorphisms.
\end{enumerate}
\end{fact}

The next results augment some computations from~\cite{sather:scc}.
In the first two, note the somewhat strange switching of supp and co-supp, as compared to Fact~\ref{cor130528a2}.

\begin{lem}\label{lem150612b}
Let $X\in\catd(R)$ and $P\in\catdfb(R)$ be such that $\pd_R(P)<\infty$.
Then there are equalities
\begin{gather*}
\supp_R(\Rhom PX)=\supp_R(P)\bigcap\supp_R(X)
\\
\cosupp_R(\Lotimes PX)=\supp_R(P)\bigcap\cosupp_R(X).
\end{gather*}
\end{lem}

\begin{proof}
Set $P^*:=\Rhom PR$.
The  assumptions $P\in\catdfb(R)$ and $\pd_R(P)<\infty$ imply that $P\simeq\Rhom{P^*}R$.
In particular, we have $P\simeq 0$ if and only if $P^*\simeq 0$. 
For any prime ideal $\p\in\spec(R)$, it follows that $P_{\p}\simeq 0$ if and only if $(P^*)_{\p}\simeq 0$,
so we have $\Supp_R(P)=\Supp_R(P^*)$, that is, $\supp_R(P)=\supp_R(P^*)$ since $P,P^*\in\catdfb(R)$.

In the next sequence of isomorphisms, the second step is Hom-evaluation~\cite[Lemma~4.4(I)]{avramov:hdouc}
$$\Rhom{P}{X}\simeq\Rhom{\Rhom{P^*}R}{X}\simeq\Lotimes{P^*}{\Rhom RX}\simeq\Lotimes{P^*}X$$
and the other steps are routine.
From this, we have
\begin{align*}
\supp_R(\Rhom PX)
&=\supp_R(\Lotimes{P^*}X) \\
&=\supp_R(P^*)\bigcap\supp_R(X) \\
&=\supp_R(P)\bigcap\supp_R(X)
\end{align*}
by Fact~\ref{cor130528a}, and hence the first equality from the statement of the result.
For the second  equality from the statement of the result, argue similarly
via the isomorphism $\Lotimes PX\simeq\Rhom{P^*}X$.
\end{proof}

\begin{lem}\label{prop151106a}
Let $X\in\catd(R)$, and let $F\in\catdb(R)$ be $\fa$-adically finite such that $\fd_R(F)<\infty$.
Then there are equalities
\begin{gather*}
\supp_R(\Rhom FX)\bigcap\VE(\fa)=\supp_R(F)\bigcap\supp_R(X)\bigcap\VE(\fa)
\\
\cosupp_R(\Lotimes FX)\bigcap\VE(\fa)=\supp_R(F)\bigcap\cosupp_R(X)\bigcap\VE(\fa).
\end{gather*}
\end{lem}

\begin{proof}
The complex $P:=\Lotimes KF$ is in $\catdfb(R)$, since $F$ is $\fa$-adically finite.
Furthermore, since $K$ and $F$ both have finite flat dimension, the same is true for $P$.
Combining these facts, we see that $P$ satisfies the hypotheses of Lemma~\ref{lem150612b}.
Using this as in the proof of~\cite[Theorem~7.12]{sather:scc}, we obtain the desired conclusions.
\end{proof}

\subsection*{Adic Semidualizing Complexes}\label{sec130818b}
The complexes defined next are introduced and studied in this generality in~\cite{sather:asc}.

\begin{disc}\label{rmk140129a}
Let $M\in\catdb(R)$
with $\supp_R(M)\subseteq V(\fa)$. 
From~\cite[Lemma 3.1]{sather:asc} we know that $M$ has a bounded above semi-injective resolution $M\xra\simeq J$ over $R$ 
consisting of injective $\Comp Ra$-modules and $\Comp{R}{a}$-module homomorphisms.
This yields a well-defined chain map $\chi^{\Comp Ra}_J\colon\Comp Ra\to\Hom JJ$ given by $\chi^{\Comp Ra}_J(r)(j)=rj$
and, in turn, a well-defined homothety morphism $\chi^{\Comp Ra}_M\colon\Comp Ra\to\Rhom MM$ in $\catd(R)$.
\end{disc}

\begin{defn}\label{def120925e}
An  \emph{$\mathfrak{a}$-adic semidualizing $R$-complex} is an $\mathfrak{a}$-adically finite $R$-complex $M$ 
(see Definition~\ref{def120925d}) such that the homothety morphism 
$\chi_{M}^{\widehat{R}^{\mathfrak{a}}}\colon \widehat{R}^{\mathfrak{a}} \rightarrow \Rhom{M}{M}$ 
from Remark~\ref{rmk140129a} is an isomorphism in $\catd(R)$.
\end{defn}

We end this section with some examples, for perspective in the sequel.

\begin{ex}
Let $M\in\catdb(R)$.
\begin{subprops}
\label{prop130528bxx}
If $M$ is an $R$-module, then it is $\fa$-adically semidualizing as an $R$-complex 
if it is $\fa$-adically finite, the natural homothety map $\Comp{R}{a} \to \Hom{M}{M}$, 
defined as in Remark~\ref{rmk140129a},
is an isomorphism, and $\Ext{i}{M}{M}=0$ for all $i \geq 1$.
\end{subprops}

\begin{subprops}
\label{prop130528b}
The complex $M$ is semidualizing if and only if it is $0$-adically semidualizing, by~\cite[Proposition~4.4]{sather:asc}. 
\end{subprops}

\begin{subprops}
\label{prop120925b}
Assume that $(R,\m)$ is local. 
Then  $M$ is $\m$-adically semidualizing if and only if each homology module $\HH_i(M)$ is artinian 
and the homothety morphism $\chi_{M}^{\widehat{R}^{\mathfrak{m}}}\colon \widehat{R}^{\mathfrak{m}} \rightarrow \Rhom{M}{M}$ is an isomorphism in $\catd(R)$.
Hence, an $R$-module $T$ is quasi-dualizing if and only if it is $\m$-adically semidualizing; see~\cite{kubik:qdm}.
Hence, the injective hull $E_R(R/\m)$ is $\m$-adically semidualizing.
See~\cite[Proposition~4.5]{sather:asc}.
\end{subprops}

\begin{subprops}
\label{thm130330b}
If $C$ is a semidualizing  $R$-complex, e.g., $C=R$, 
then the complex $\RG{\fa}{C}$ is $\fa$-adically semidualizing, by~\cite[Corollary~4.8]{sather:asc}.
\end{subprops}
\end{ex}

\section{Foxby Classes}\label{sec130818d}

This section develops the foundations of Auslander and Bass classes in the adic context. 
It contains our version of the ubiquitous ``Foxby Equivalence'', which is Theorem~\ref{thm121116ax} from the introduction, among other results.

\begin{defn}\label{defn121104a}
Let $M,X,Y\in\catdb(R)$. 
\begin{enumerate}[(a)]
\item
The complex $X$ is in the \textit{Auslander class} $\mathcal{A}_M(R)$ if $\Lotimes{M}{X}\in\catdb(R)$ 
and the natural morphism $\gamma_X^M\colon X \rightarrow \Rhom{M}{\Lotimes{M}{X}}$ is  an isomorphism in $\catd(R)$.
\item
The complex $Y$ is in the \textit{Bass class} $\mathcal{B}_M(R)$ if one has $\Rhom{M}{Y}\in\catdb(R)$  and 
the natural evaluation morphism $\xi_Y^M\colon \Lotimes{M}{\Rhom{M}{Y}} \rightarrow Y$ is an isomorphism in $\catd(R)$.
\end{enumerate}
\end{defn}

The next result gives some  examples of objects in Foxby classes to keep in mind.
See Propositions~\ref{prop140218a} and~\ref{prop150527a}\eqref{prop150527a1} 
for improvements on the conclusion $\Comp{R}{a} \in \mathcal{A}_M(R)$.

\begin{prop}\label{prop130612a}
Let $M$ be an $\fa$-adic semidualizing $R$-complex.
Then one has $\Comp{R}{a} \in \mathcal{A}_M(R)$ and $M \in \mathcal{B}_M(R)$.
\end{prop}

\begin{proof}
First, we show that $\Comp{R}{a} \in \mathcal{A}_M(R)$.
Since $\Comp Ra$ is flat over $R$ and $M\in\catdb(R)$, we have $\Lotimes{M}{\Comp Ra}\in\catdb(R)$.
By~\cite[Theorem~5.10]{sather:scc}, the natural morphism $\alpha\colon M\to \Lotimes M{\Comp Ra}$ is an isomorphism in $\catd(R)$.
From  the next commutative diagram in $\catd(R)$
$$\xymatrix{
\Comp{R}{a} \ar[r]^-{\chi_{M}^{\Comp{R}{a}}}_-{\simeq} \ar[rd]_-{\gamma_{\Comp{R}{a}}^M} & \Rhom{M}{M}\ar[d]^{\Rhom{M}{\alpha}}_{\simeq}
  \\
& \Rhom{M}{\Lotimes{M}{\Comp{R}{a}}} }$$
we conclude that $\gamma_{\Comp{R}{a}}^M$ is an isomorphism, so $\Comp{R}{a} \in \mathcal{A}_M(R)$.

Next, we show that $M \in \mathcal{B}_M(R)$.
By assumption, the homothety morphism $\chi^{\Comp Ra}_M\colon\Comp Ra\to\Rhom{M}{M}$ is an isomorphism in $\catd(R)$. In particular,
we have $\Rhom MM\in\catdb(R)$.
The composition of the following morphisms
$$M\xra[\simeq]{\alpha}
\Lotimes{M}{\Comp{R}{a}} \xra[\simeq]{\Lotimes{M}{\chi_{M}^{\Comp{R}{a}}}} \Lotimes{M}{\Rhom{M}{M}} \xra{\xi_M^M} M$$
is $\id_M$, so we conclude that $\xi_M^M$ is an isomorphism, so $M \in \mathcal{B}_M(R)$.
\end{proof}

Much of this work highlights the similarities between
Christensen's setting~\cite{christensen:scatac} where $\fa=0$ and the general case.
However, the next two items document some important differences to keep in mind.

\begin{fact}[\protect{\cite[Observation 4.10]{christensen:scatac}}]
\label{rmk140218b}
Let $C$ be a semidualizing $R$-module.
\begin{subprops}
An $R$-module $A$ is in  $\mathcal{A}_C(R)$ if and only if 
the natural map
$\gamma^M_A\colon A\to \Hom{M}{\Otimes{M}{A}}$  is an isomorphism
and for all $i\geq 1$ we have
$\Tor{i}{M}{A}=0=\Ext{i}{M}{\Otimes{M}{A}}$.
\end{subprops}
\begin{subprops}
An $R$-module $B$ is in $\mathcal{B}_C(R)$ if and only if the natural evaluation homomorphism
$\xi^M_B\colon \Otimes{M}{\Hom{M}{B}} \to B$  is an isomorphism and for all $i\geq 1$ we have $\Ext{i}{M}{B}=0=\Tor{i}{M}{\Hom{M}{B}}$ .
\end{subprops}
\end{fact}

\begin{ex}\label{ex130504b}
Let $k$ be a field, and let $R=k[\![X]\!]$ be a power series ring in one variable. Let $E$ be the injective hull $E_R(k)$.
Example~\ref{prop120925b} implies that $E$ is $\m$-adically semidualizing over $R$. 

For this example, we define $\mathcal{A}^0_E(R)$ to be the class of all
$R$-modules $A$ such that 
the natural map
$\gamma^E_A\colon A\to \Hom{E}{\Otimes{E}{A}}$ is an isomorphism and
for all $i\geq 1$ we have
$\Tor{i}{E}{A}=0=\Ext{i}{E}{\Otimes{E}{A}}$.
Define $\catb^0_E(R)$ similarly.

It is straightforward to show that $R$ is in $\cata^0_E(R)$ in this example.
Based on work in the semidualizing case, one may expect $k=R/XR$ 
to be in $\mathcal{A}^0_{E}(R)$, as it is a module with finite flat dimension. However, this module fails the definition of $\mathcal{A}^{0}_{E}(R)$ in two ways. 

First, we have $\Hom{E}{\Otimes{E}{k}}\cong \Hom{E}{0} =0$; so it is not possible for 
the natural map $\gamma^E_k\colon k\to \Hom{E}{\Otimes{E}{k}}$ to be an isomorphism.
Second, we have $\Tor{1}{E}{k} \cong k$ and $\Tor{i}{E}{k}=0$ for all $i \neq 1$; this is straightforward to show using the Koszul complex $K^R(X)$ as a free resolution of $k$.

This example is even more troubling because it shows that $\cata^0_E(R)$ does not satisfy the 2-of-3 condition. Indeed, the following is an exact sequence
$$\xymatrix{
 0 \to R \xra X  R \to  k \to 0
}$$
and the first two modules are in $\cata_E(R)$, but the third is not.

Similarly, by dualizing the above exact sequence with respect to $E$, one obtains an augmented injective resolution of $k$
$$0\to k\to E\xra XE\to 0.$$
From this, we see that $\Ext 1Ek\cong k$ and $\Ext iEk=0$ for all $i\neq 1$.
As above, this shows that $k\notin\catb^0_E(R)$ and that $\catb^0_E(R)$ does not satisfy the 2-of-3 condition.
\end{ex}

The next result shows that the classes $\cata_M(R)$ and $\catb_M(R)$ do not have the same flaws as 
the classes from the previous example.

\begin{prop}\label{prop140112a}
Let $M$ be an $\mathfrak{a}$-adic semidualizing $R$-complex. 
Then the classes $\cata_M(R)$ and $\catb_M(R)$ are triangulated and thick.
\end{prop}

\begin{proof}
By definition, this follows from the next straightforward facts:
\begin{enumerate}[1.]
\item For each $i\in\bbz$, the classes $\cata_M(R)$ and $\catb_M(R)$ are closed under $\shift^i$.
\item
Given an exact triangle $X\to Y\to Z\to$ in $\catd(R)$,
if two of the three complexes $X,Y,Z$ are in $\cata_M(R)$ (repectively, in $\catb_M(R)$), then so is the third.
\item
For all $X,Y\in\catd_{\text b}(R)$, the direct sum
$X\oplus Y$ is in $\cata_M(R)$ if and only if 
$X$ and $Y$ are both in $\cata_M(R)$, and similarly for $\catb_M(R)$.\qedhere
\end{enumerate}
\end{proof}

Next, we prove the adic version of Foxby Equivalence, which is Theorem ~\ref{thm121116ax} in the introduction.
Note the support and co-support conditions in parts~\eqref{thm121116a2} and~\eqref{thm121116a3}, which are automatic in the semidualizing 
situation~\cite[Theorem 4.6]{christensen:scatac}.
Example~\ref{ex130530a} below shows that they are crucial in our more general setup. 
Note also the lack of any \emph{a priori} boundedness condition in parts~\eqref{thm121116a2} and~\eqref{thm121116a3}.

\begin{thm}\label{thm121116a}
Let $M$ be an $\mathfrak{a}$-adic semidualizing $R$-complex. 
\begin{enumerate}[\rm(a)]
\item\label{thm121116a1}
The functors $\Rhom{M}{-}: \mathcal{B}_M(R) \rightarrow \mathcal{A}_M(R)$ and $\Lotimes{M}{-}: \mathcal{A}_M(R) \rightarrow \mathcal{B}_M(R)$ are quasi-inverse equivalences. \item\label{thm121116a2}
An $R$-complex $Y\in\catd(R)$ is in $\mathcal{B}_{M}(R)$ if and only if $\Rhom{M}{Y} \in  \mathcal{A}_M(R)$ and $\supp_{R}(Y) \subseteq \VE(\fa)$.
\item\label{thm121116a3}
An $R$-complex  $X\in\catd(R)$ is in $\mathcal{A}_M(R)$ if and only if one has $\Lotimes{M}{X} \in  \mathcal{B}_M(R)$ and $\operatorname{co-supp}_{R}(X) \subseteq \VE(\fa)$.
\end{enumerate}
\end{thm}

\begin{proof}
Let $X\in\catdb(R)$, and set $Z:=\Lotimes{M}{X}$. Consider the defining morphisms $\xi_Z^M\colon \Lotimes{M}{\Rhom{M}{Z}} \rightarrow Z$ and 
$\gamma_X^M\colon X \rightarrow \Rhom{M}{Z}$. Then the induced morphism in $\catd(R)$
$$\xymatrix{
Z=\Lotimes{M}{X} \ar[rr]^-{\Lotimes{M}{\gamma_X^M}} && \Lotimes{M}{\Rhom{M}{Z}}
}$$
satisfies $\xi_Z^M\circ (\Lotimes{M}{\gamma_X^M})=\id_Z$. It follows that $\Lotimes{M}{\gamma_X^M}$ is an isomorphism if and only if $\xi_Z^M$ is one.

We verify the forward implication of part~\eqref{thm121116a3}.
To this end, assume for this paragraph  
that $X \in \mathcal{A}_M(R)$. Then we have $Z,\Rhom{M}{Z}\in\catdb(R)$, and ${\gamma_X^M}$ is an isomorphism
$X \xra\simeq \Rhom{M}{\Lotimes{M}{X}}$. Thus, Fact~\ref{cor130528a2} implies that
$$\operatorname{co-supp}_{R}(X) = \operatorname{co-supp}_{R}(\Rhom{M}{\Lotimes{M}{X}}) \subseteq\supp_R(M)\subseteq\VE(\fa).$$
The morphism $\Lotimes{M}{\gamma_X^M}$ is also an isomorphism, since ${\gamma_X^M}$ is one. From the previous paragraph, we know that $\xi_Z^M$ is also an isomorphism, and therefore, we have $Z \in \mathcal{B}_M(R)$, as desired. 

We now prove the converse of part~\eqref{thm121116a3}.
Assume that $Z=\Lotimes{M}{X} \in \mathcal{B}_M(R)$ and $\operatorname{co-supp}_{R}(X) \subseteq \VE(\fa)$. Then we have $Z,\Rhom{M}{Z}\in\catdb(R)$, 
and $\xi_Z^M$ is an isomorphism. 
By the first paragraph of this proof, the morphism $\Lotimes{M}{\gamma_X^M}$ is therefore an isomorphism. Fact~\ref{cor130528a2} implies 
$\cosupp_{R}(\Rhom{M}{Z}) \subseteq \VE(\fa)$. Since we also have
$\cosupp_R(X)\subseteq\VE(\fa)$, we conclude from~\cite[Theorem~5.7]{sather:scc}  that $\gamma_X^M$ is an isomorphism
in $\catd(R)$. Hence, we have $X\simeq\Rhom MZ\in\catdb(R)$.
As we also have $\Lotimes MX=Z\in\catdb(R)$, we conclude that $X \in \mathcal{A}_M(R)$, as desired.

Part~\eqref{thm121116a2} is verified similarly, and part~\eqref{thm121116a1} follows from~\eqref{thm121116a2} and~\eqref{thm121116a3}.
\end{proof}

Next, we show the necessity of the support conditions in Foxby Equivalence~\ref{thm121116a}.

\begin{ex}\label{ex130530a}
Let $k$ be a field, and set $R:=k[\![Y]\!]$ with $E:=E_R(k)$ and $\fa:=(Y)R$.
We show that $R\notin\catb_E(R)$ and $\Rhom{E}{R} \in\cata_E(R)$.
We also show that $E\notin\cata_E(R)$ and $\Lotimes EE\in\catb_E(R)$.
Notice that these facts do not contradict 
Foxby Equivalence~\ref{thm121116a}, because the support condition is not satisfied:
by faithful flatness and faithful injectivity we have $\supp_{R}(R) 
=\cosupp_R(E)=\spec(R)\not\subseteq \VE(\fa)$.

We first show $\Rhom{E}{R} \in \mathcal{A}_{E}(R)$. The augmented and truncated minimal semi-injective resolutions of $R$ are,
respectively, 
$$\xymatrix@R=1mm{
^{+}J = \quad 0 \ar[r] & R \ar[r] & Q(R) \ar[r] &E \ar[r] &0 \\
J = \quad  & 0 \ar[r] & Q(R) \ar[r] &E \ar[r] &0
}$$
where $Q(R)=k(\!(Y)\!)$ is the field of fractions of $R$.
An application of $\Hom{E}{-}$ to $J$ yields the complex
$$\xymatrix@C=8mm{
\Hom{E}{J} = \quad 0 \ar[r] & \Hom{E}{Q(R)} \ar[r] & \Hom{E}{E} \ar[r] & 0.
}$$
It is well-known that $\Hom{E}{Q(R)}= 0$ and $\Hom{E}{E} \cong R$. It follows that $\Rhom{E}{R} \simeq \Hom{E}{J}\simeq \shift^{-1}R\in\cata_E(R)$,
by Proposition~\ref{prop130612a}.

Next, consider the isomorphisms $$\shift^{-1}E \simeq \Lotimes{E}{\shift^{-1} R} \simeq \Lotimes{E}{\Rhom{E}{R}}.$$
It follows that the morphism $\delta^{E}_{R}\colon \Lotimes{E}{\Rhom{E}{R}} \to R$ is not isomorphism. Hence, we have $R \not\in \mathcal{B}_{E}(R)$. One can also deduce this from Foxby Equivalence~\ref{thm121116a}\eqref{thm121116a2} since we have $\supp_{R}(R) 
=\spec(R)\not\subseteq \VE(\fa)$.

Next, 
it is straightforward to show that
$\Lotimes EE\simeq\shift^1E\in\catb_E(R)$; see~\cite[Example~6.4]{kubik:hamm1} and Proposition~\ref{prop130612a}.
From this, we have
$$\Rhom E{\Lotimes EE}\simeq\Rhom E{\shift^1E}\simeq\shift^1\Rhom EE\simeq\shift^1R\not\simeq E$$
so $E\notin\cata_E(R)$; one can also deduce this using $\cosupp_R(E)$ as above.
\end{ex}

The next result shows what happens when you do remove the support conditions from Foxby Equivalence~\ref{thm121116a}.
Again, note the lack of boundedness assumptions. 

\begin{cor}\label{cor151221a}
Let $M$ be an $\fa$-adic semidualizing $R$-complex, and let $X\in\catd(R)$. 
\begin{enumerate}[\rm(a)]
\item\label{cor151221a1}
One has $\Rhom{M}{X} \in  \mathcal{A}_M(R)$ if and only if $\RG aX\in\catb_M(R)$.
When these conditions are satisfied, one has $\RG aX\simeq\Lotimes M{\Rhom MX}$.
\item\label{cor151221a2}
One has $\Lotimes{M}{X} \in  \mathcal{B}_M(R)$ if and only if $\LL aX\in\cata_M(R)$.
When these conditions are satisfied, one has $\LL aX\simeq\Rhom M{\Lotimes MX}$.
\end{enumerate}
\end{cor}

\begin{proof}
We prove part~\eqref{cor151221a1}. 
From Fact~\ref{fact151222c}, we have
\begin{equation}\label{eq151223a}
\Rhom M{\RG aX}\simeq\Rhom MX.
\end{equation}

For the forward implication, assume  $\Rhom{M}{X} \in  \mathcal{A}_M(R)$.
From~\eqref{eq151223a}, we have
$\Rhom M{\RG aX}\in\cata_M(R)$.
Fact~\ref{cor130528a2} implies that $\supp_R(\RG aX)\subseteq\VE(\fa)$, so Foxby Equivalence~\ref{thm121116a}\eqref{thm121116a2} implies that
$\RG aX\in\catb_M(R)$.

For the converse, assume that $\RG aX\in\catb_M(R)$.
By definition, this yields the second  isomorphism in the next sequence.
$$\Lotimes M{\Rhom MX}\simeq\Lotimes M{\Rhom M{\RG aX}}\simeq\RG aX\in\catb_M(R)$$
The first isomorphism is again by~\eqref{eq151223a}.
Fact~\ref{cor130528a2} implies 
$$\cosupp_R(\Rhom MX)\subseteq\supp_R(M)\subseteq\VE(\fa)$$
so by Foxby Equivalence~\ref{thm121116a}\eqref{thm121116a2}, we have 
$\Rhom MX\in\cata_M(R)$, as desired. 
\end{proof}

\begin{ex}\label{ex151223a}
Let $k$ be a field, and set $R:=k[\![Y]\!]$ with $E:=E_R(k)$ and $\fa:=(Y)R$.
From Example~\ref{ex130530a} we have $R\notin\catb_E(R)$ and $\Rhom{E}{R} \in\cata_E(R)$.
On the other hand,  Corollary~\ref{cor151221a}\eqref{cor151221a1} implies that
$$\shift^{-1}E\simeq\RG aR\simeq\Lotimes E{\Rhom ER}\in\catb_E(R).$$
Note that this corroborates part of Proposition~\ref{prop130612a}.
Example~\ref{ex130530a} also shows that $E\notin\cata_E(R)$ and $\Lotimes EE\in\catb_E(R)$.
Corollary~\ref{cor151221a}\eqref{cor151221a1} implies that
$$\shift^1R\simeq\Rhom E{\Lotimes EE}\simeq\LL aE\in\cata_E(R)$$
which again bears witness to Proposition~\ref{prop130612a}.
\end{ex}

Our next results document adic versions of some standard facts,
starting with an augmentation of Proposition~\ref{prop130612a}.

\begin{prop}\label{prop140218a}
Let $M$ be an $\fa$-adic semidualizing $R$-complex. 
Then one has $R\in\cata_M(R)$ if and only if $R$ is $\fa$-adically complete.
\end{prop}

\begin{proof}
For the forward implication, assume that we have $R\in\cata_M(R)$.
Foxby Equivalence~\ref{thm121116a}\eqref{thm121116a3} implies that $\cosupp_R(R)\subseteq\VE(\fa)$,
so  $R\simeq\LL aR\simeq\Comp Ra$ by Fact~\ref{cor130528a3}, thus $R$ is $\fa$-adically complete.
Conversely, if $R$ is $\fa$-adically complete, then $R\cong\Comp Ra\in\cata_M(R)$ by Proposition~\ref{prop130612a}.
\end{proof}

The next remark is for use in the sequel.

\begin{disc}\label{disc150527a}
Let $M$ be an $\fa$-adic semidualizing $R$-complex, and let $X,Y\in\catd(R)$.
Since $\Comp Ra$ is flat over $R$ and the homology of $K$ is $\fa$-torsion,
the natural morphism $\iota\colon K\to \Lotimes K{\Comp Ra}$ is an isomorphism in $\catd(R)$.
This explains the  vertical isomorphism in  the following commutative diagram in $\catd(R)$:
\begin{equation}
\label{eq150527a}
\begin{split}
\xymatrix@C=20mm{
\Lotimes KX
\ar[r]^-{\Lotimes{K}{\gamma^M_X}}\ar[d]_{\Lotimes \iota X}^\simeq
&\Lotimes K{\Rhom M{\Lotimes MX}} \\
\Lotimes{\Lotimes K{\Comp Ra}}X
\ar[r]^-{\Lotimes{\Lotimes K{\chi^{\Comp Ra}_M}}X}_-\simeq
&\Lotimes{\Lotimes K{\Rhom MM}}X
\ar[u]_{\Lotimes K{\omega_{MMX}}}.
}
\end{split}
\end{equation}
The morphism $\omega_{MMX}$ is tensor-evaluation.
The lower horizontal morphism is an isomorphism since $M$ is $\fa$-adically semidualizing.
Similarly, we have the next commutative diagram
where $\nu:=\Rhom K{\Rhom{\chi^{\Comp Ra}_M}{Y}}$.
\begin{equation}
\label{eq150527b}
\begin{split}
\xymatrix@C=4mm{
\Rhom K{\Lotimes M{\Rhom MY}}\ar[r]^-{\Rhom K{\xi^M_Y}}\ar[dd]_{\Rhom K{\theta_{MMY}}}
&\Rhom KY
\\
&\Rhom{\Lotimes{\Comp Ra}K}Y\ar[u]^\simeq_{\Rhom\iota Y} \\
\Rhom K{\Rhom{\Rhom MM}Y}
\ar[r]_-\simeq^-{\nu}
&\Rhom K{\Rhom{\Comp Ra}Y}\!\!\!\!\!\!\!\!\!
\ar[u]^\simeq
}
\end{split}
\end{equation}
The unspecified  isomorphism is adjointness,
and  $\theta_{MMY}$ is tensor-evaluation.
\end{disc}

It is straightforward to show that the trivial semidualizing complex $R$ has trivial Auslander and Bass classes:
$\cata_R(R)=\catdb(R)=\catb_R(R)$.
Our next result generalizes this to the adic  situation.
Note that Foxby Equivalence~\ref{thm121116a} shows that this is as trivial as things get in this setting.
Also, see~\cite[Section~5]{sather:asc} for characterizations of the property $\fd_R(M)<\infty$ in this context.

\begin{prop}\label{prop150525b}
Let $M$ be an $\fa$-adic semidualizing complex with $\fd_R(M)<\infty$, e.g., $M=\RG aR$.
\begin{enumerate}[\rm(a)]
\item\label{prop150525b1}
The class $\cata_M(R)$ consists of all $X\in\catdb(R)$ such that $\cosupp_R(X)\subseteq\VE(\fa)$, 
that is, the complexes $X\simeq\LL aZ$ for some $Z\in\catdb(R)$.
\item\label{prop150525b2}
The class $\catb_M(R)$ consists of all $Y\in\catdb(R)$ with $\supp_R(Y)\subseteq\VE(\fa)$,
that is, the complexes $X\simeq\RG aZ$ for some $Z\in\catdb(R)$.
\end{enumerate}
\end{prop}

\begin{proof}
MGM Equivalence~\cite[Theorem~6.11]{yekutieli:hct} shows that 
$X\in\catdb(R)$ satisfies $\cosupp_R(X)\subseteq\VE(\fa)$ if and only if
we have $X\simeq\LL aZ$ for some $Z\in\catdb(R)$, and similarly for supp.
(Parts of this are in Facts~\ref{cor130528a2}--\ref{cor130528a3}.)
Thus, by Foxby Equivalence~\ref{thm121116a}, we need only show the following:
(1) if $X\in\catdb(R)$ satisfies $\cosupp_R(X)\subseteq\VE(\fa)$, then $X\in\cata_M(R)$, and
(2) if $Y\in\catdb(R)$ satisfies $\supp_R(Y)\subseteq\VE(\fa)$, then $Y\in\catb_M(R)$.

(1) Let $X\in\catdb(R)$ be such that $\cosupp_R(X)\subseteq\VE(\fa)$.
Since $\fd_R(M)<\infty$, we have $\Lotimes MX\in\catdb(R)$.
Also, Fact~\ref{cor130528a2} implies that
$$\cosupp_R(\Rhom{M}{\Lotimes MX})\subseteq\supp_R(M)\subseteq\VE(\fa).$$
Thus, by Fact~\ref{thm130318aqqc}, to show that $\gamma^M_X$ is an isomorphism,
it suffices to show that $\Lotimes{K}{\gamma^M_X}$ is an isomorphism.
The morphism $\Lotimes K{\omega_{MMX}}$ from Remark~\ref{disc150527a}
is an isomorphism by Fact~\ref{fact151222a}\eqref{fact151222a1}.
It follows from the diagram~\eqref{eq150527a} that $\Lotimes{K}{\gamma^M_X}$ is an isomorphism, as desired.

(2) Let $Y\in\catdb(R)$ be such that $\supp_R(Y)\subseteq\VE(\fa)$.
Since $M$ has finite projective dimension by~\cite[Theorem~6.1]{sather:afcc}, we have $\Rhom MY\in\catdb(R)$. 
Thus, to complete the proof, one argues as in part~(1) to show that the morphism $\xi^M_Y$ is an isomorphism in $\catd(R)$,
using~\eqref{eq150527b} and Fact~\ref{fact151222a}\eqref{fact151222a2}.
\end{proof}

Given a semidualizing $R$-complex $C$, we know from~\cite[Proposition 4.4]{christensen:scatac}  that $\catac(R)$ contains all complexes of finite
flat dimension, and $\catbc(R)$ contains all complexes of finite injective dimension. 
The next result is our version of this fact in the adic setting. 

\begin{prop}\label{prop150527a}
Let $M$ be an $\fa$-adic semidualizing $R$-complex. 
\begin{enumerate}[\rm(a)]
\item \label{prop150527a1}
The Auslander class $\cata_M(R)$ contains all $R$-complexes 
$X\in\catdb(R)$ of finite flat dimension such that $\cosupp_R(X)\subseteq\VE(\fa)$.
\item \label{prop150527a2}
The Bass class $\catb_M(R)$ contains all $R$-complexes $Y\in\catdb(R)$ of finite injective dimension such that $\supp_R(Y)\subseteq\VE(\fa)$.
\end{enumerate}
\end{prop}

\begin{proof}
We deal with part~\eqref{prop150527a1}.
Let $X\in\catdb(R)$ with $\fd_R(X)<\infty$ be such that $\cosupp_R(X)\subseteq\VE(\fa)$.
The condition $\fd_R(X)<\infty$ implies that $\Lotimes MX\in\catdb(R)$.
To show that the morphism
$\gamma^M_X\colon X\to \Rhom M{\Lotimes MX}$ is an isomorphism in $\catd(R)$,
it suffices by Fact~\ref{thm130318aqqc} to show that the induced morphism
$\Lotimes K{\gamma^M_X}$ is an isomorphism.
This is accomplished
using Fact~\ref{fact151222a}\eqref{fact151222a1} with the diagram~\eqref{eq150527a} from Remark~\ref{disc150527a}, as in the proof of 
Proposition~\ref{prop150525b}.
\end{proof}

\begin{cor}\label{cor150527a}
Let $M$ be an $\fa$-adic semidualizing $R$-complex, and fix an ideal $\fb\supseteq\fa$. 
Let $L$  be the Koszul complex over $R$ on a finite generating sequence for $\fb$.
\begin{enumerate}[\rm(a)]
\item \label{cor150527a1}
The Auslander class $\cata_M(R)$ contains every $R$-complex of the form $\LL bF$ and $\Lotimes LF$ where 
$F\in\catdb(R)$ has finite flat dimension.
In particular, we have
$\Comp Rb,L\in\cata_M(R)$.
\item \label{cor150527a2}
The Bass class $\catb_M(R)$ contains 
every complex of the form $\RG bI$ and $\Lotimes LI$ where $I\in\catdb(R)$ has finite injective dimension.
\end{enumerate}
\end{cor}

\begin{proof}
We prove part~\eqref{cor150527a1}.
The complex $\LL bF$ has finite flat dimension, by Fact~\ref{fact130619b}.
Since we also have $\cosupp_R(\LL bF)\subseteq\VE(\fb)\subseteq\VE(\fa)$ by Fact~\ref{cor130528a2},
we deduce from Proposition~\ref{prop150527a}\eqref{prop150527a1} that $\LL bF\in\cata_M(R)$,
and similarly for $\Lotimes LF$, using Fact~\ref{cor130528a3}.
The conclusion $\Comp Rb,L\in\cata_M(R)$ is from the special case $F=R$.
\end{proof}

Our next three results are versions of~\cite[Corollary 2.10 and Theorem 2.11]{takahashi:hasm}
and~\cite[Theorem 1.1]{totushek:hdsc} for the adic semidualizing context.
As with the previous results, a major difference is the inclusion of a (co)support condition.

\begin{prop}\label{prop150527b}
Let $M$ be  $\fa$-adic semidualizing over $R$, and let $X\in\catdb(R)$. 
The following conditions are equivalent.
\begin{enumerate}[\rm(i)]
\item \label{prop150527b1}
One has $\fd_R(\Rhom MX)<\infty$ and $\supp_R(X)\subseteq\VE(\fa)$. 
\item \label{prop150527b2}
There is a complex $F\in\catdb(R)$ with $\fd_R(F)<\infty$ such that $X\simeq\Lotimes MF$ and $\cosupp_R(F)\subseteq\VE(\fa)$.
\item \label{prop150527b3}
There is a complex $G\in\catdb(R)$ with $\fd_R(G)<\infty$ such that $X\simeq\Lotimes MG$.
\end{enumerate}
If these conditions hold, then $F\simeq\Rhom MX\simeq\LL aG$ and $X\in\catb_M(R)$.
\end{prop}

\begin{proof}
\eqref{prop150527b1}$\implies$\eqref{prop150527b3}.
Assume  that $\supp_R(X)\subseteq\VE(\fa)$ and that the complex $G:=\Rhom MX\in\catdb(R)$ satisfies $\fd_R(G)<\infty$.
Since Fact~\ref{cor130528a2}  implies that
$$\cosupp_R(G)=\cosupp_R(\Rhom MX)\subseteq\supp_R(M)\subseteq\VE(\fa)$$
we know from Proposition~\ref{prop150527a}\eqref{prop150527a1} that
$G=\Rhom MX\in\cata_M(R)$. The condition   $\supp_R(X)\subseteq\VE(\fa)$ implies that $X\in\catb_M(R)$, by
Foxby Equivalence~\ref{thm121116a}\eqref{thm121116a2}.
Thus, we have $X\simeq \Lotimes M{\Rhom MX}=\Lotimes MG$, as desired.

\eqref{prop150527b3}$\implies$\eqref{prop150527b2}.
Assume that there is a complex $G\in\catdb(R)$ with $\fd_R(G)<\infty$ such that $X\simeq\Lotimes MG$.
Set $F:=\LL aG$, which satisfies $\fd_R(F)<\infty$ and $\cosupp_R(F)\subseteq\VE(\fa)$, 
by Facts~\ref{fact130619b} and~\ref{cor130528a2}. 
The first and last isomorphisms in the next sequence are by assumption
$$\Lotimes M{F}
\simeq\Lotimes M{\LL aG}
\simeq\Lotimes MG
\simeq X$$
and the second isomorphism is from Fact~\ref{fact151222c}.
Also, in the next sequence, the first isomorphism is by definition and the last one is from the previous display
$$\LL aG\simeq F\simeq\Rhom{M}{\Lotimes MF}\simeq\Rhom MX.$$
The second isomorphism is because $F\in\cata_M(R)$; see Proposition~\ref{prop150527a}\eqref{prop150527a1}.
This completes the proof of this implication and explains one of the additional claims
in the statement of the proposition. 

\eqref{prop150527b2}$\implies$\eqref{prop150527b1}. 
Assume that there is a complex $F\in\catdb(R)$ with $\fd_R(F)<\infty$ such that $X\simeq\Lotimes MF$ and $\cosupp_R(F)\subseteq\VE(\fa)$.
Proposition~\ref{prop150527a}\eqref{prop150527a1} implies that $F\in\cata_M(R)$,
so we have $X\simeq\Lotimes MF\in\catb_M(R)$ by Foxby Equivalence~\ref{thm121116a}\eqref{thm121116a2}.
Moreover, this implies that we have
$$F\simeq\Rhom{M}{\Lotimes MF}\simeq\Rhom MX$$
so $\fd_R(\Rhom MX)=\fd_R(F)<\infty$. The isomorphism $X\simeq\Lotimes MF$ implies 
$$\supp_R(X)=\supp_R(\Lotimes MF)\subseteq\supp_R(M)\subseteq\VE(\fa)$$
by Fact~\ref{cor130528a2}.
This completes the proof of this implication and explains the remaining claims
in the statement of the proposition. 
\end{proof}

The next two results are proved similarly to the previous one.

\begin{prop}\label{prop150527c}
Let $M$ be an $\fa$-adic semidualizing $R$-complex, and let $X\in\catdb(R)$. 
The following conditions are equivalent.
\begin{enumerate}[\rm(i)]
\item \label{prop150527c1}
One has $\pd_R(\Rhom MX)<\infty$ and $\supp_R(X)\subseteq\VE(\fa)$. 
\item \label{prop150527c2}
There is a complex $P\in\catdb(R)$ with $\pd_R(P)<\infty$ such that $X\simeq\Lotimes MP$ and $\cosupp_R(P)\subseteq\VE(\fa)$.
\item \label{prop150527c3}
There is a complex $Q\in\catdb(R)$ with $\pd_R(Q)<\infty$ such that $X\simeq\Lotimes MQ$.
\end{enumerate}
If these conditions hold, then $P\simeq\Rhom MX\simeq\LL aQ$ and $X\in\catb_M(R)$.
\end{prop}

\begin{prop}\label{prop150527d}
Let $M$ be an $\fa$-adic semidualizing  $R$-complex, and let $Y\in\catdb(R)$. 
The following conditions are equivalent.
\begin{enumerate}[\rm(i)]
\item \label{prop150527d1}
One has $\id_R(\Lotimes MX)<\infty$ and $\cosupp_R(Y)\subseteq\VE(\fa)$. 
\item \label{prop150527d2}
There is a complex $I\in\catdb(R)$ with $\id_R(I)<\infty$ such that $Y\simeq\Rhom MI$ and $\supp_R(I)\subseteq\VE(\fa)$.
\item \label{prop150527d3}
There is a complex $J\in\catdb(R)$ with $\id_R(J)<\infty$ such that $Y\simeq\Rhom MJ$.
\end{enumerate}
When these conditions hold, one has $I\simeq\Lotimes MY\simeq\RG aJ$ and $Y\in\cata_M(R)$.
\end{prop}

We conclude this section a technical, but useful result, \emph{\`a la}~\cite[Proposition~4.8]{christensen:scatac}.

\begin{lem}\label{lem151126a}
Let $M$ be an $\mathfrak{a}$-adic semidualizing $R$-complex, and let $X\in\catdb(R)$.
\begin{enumerate}[\rm(a)]
\item \label{lem151126a1}
If $X\in\cata_M(R)$, then we have
$$\inf(M)+\sup(X)\leq\sup(\Lotimes MX)\leq\sup(X)+\sup(M)+n.$$
\item \label{lem151126a2}
If $X\in\catb_M(R)$, then we have
$$\inf(X)-\sup(M)-2n\leq\inf(\Rhom MX)\leq\inf(X)-\inf(M).$$
\end{enumerate}
\end{lem}

\begin{proof}
\eqref{lem151126a1}
The condition $X\in\cata_M(R)$ implies that $X\simeq\Rhom M{\Lotimes MX}$.
This explains the first step in the next display.
\begin{align*}
\sup(X)
&=\sup(\Rhom M{\Lotimes MX}) 
\geq\sup(\Lotimes MX)-\sup(M)-n
\end{align*}
The second step follows from~\cite[Proposition~3.1(a)]{sather:afbha}; the hypotheses of this result are satisfied since we have
$\Lotimes MX\in\catdb(R)$, and 
$M$ is $\fa$-adically finite with $\supp_R(M)=\VE(\fa)\supseteq\supp_R(\Lotimes MX)$ by~\cite[Proposition~7.17]{sather:scc}.
This yields the second of the inequalities from the statement of the lemma. 
For the first one, we argue similarly, using Fact~\ref{disc151112a}\eqref{disc151112a5}:
\begin{align*}
\sup(X)
&=\sup(\Rhom M{\Lotimes MX}) 
\leq\sup(\Lotimes MX)-\inf(M)
\end{align*}

\eqref{lem151126a2}
These inequalities are verified similarly, using~\cite[Proposition~3.6(a)]{sather:afbha}.
\end{proof}

\section{Stability}\label{sec130818dz}

In this section, we document various stability results (and special converses) for Foxby classes, including Theorem~\ref{prop151123ar}
from the introduction. 

\subsection*{Sums and Products}

Let $M$ be an $\mathfrak{a}$-adic semidualizing $R$-complex. 
Then $\cata_M(R)$ and $\catb_M(R)$ always fail to be closed under arbitrary direct sums and products. 
The main reason for this is that if $N_i\in\catdb(R)$, then in general we have $\bigoplus_iN_i,\prod_iN_i\notin\catdb(R)$.
However, the next example shows that this can fail for other reasons.

\begin{ex}\label{ex151125a}
Let $k$ be a field, and set $R:=k[\![Y]\!]$ with $E:=E_R(k)$ and $\fa:=(Y)R$.
Then we have $R\in\cata_E(R)$ and $E\in\catb_E(R)$ by Proposition~\ref{prop130612a}.

We claim that the direct sum $R^{(\bbn)}$ is not contained in $\cata_E(R)$. 
Indeed, from~\cite[Lemma~3.2]{sather:elclh} we know that the natural map $R^{(\bbn)}\to\Lambda^{\fa}(R^{(\bbn)})$ is not an isomorphism.
Since $R^{(\bbn)}$ is flat, this says that the natural morphism $R^{(\bbn)}\to\LL m{R^{(\bbn)}}$ is not an isomorphism in $\catd(R)$,
so $\cosupp_R(R^{(\bbn)})\not\subseteq\VE(\fa)$ by Fact~\ref{cor130528a3}. 
Thus, Foxby Equivalence~\ref{thm121116a}\eqref{thm121116a3} shows that
$R^{(\bbn)}\notin\cata_E(R)$.

Similarly, the fact that  the product $E^{\bbn}$ is not $\fa$-torsion implies that $\supp_R(E^{\bbn})\not\subseteq\VE(\fa)$, 
so
$E^{\bbn}\notin\catb_E(R)$.

On the other hand, it is straightforward to show that we do have $R^{\bbn}\in\cata_E(R)$ and $E^{(\bbn)}\in\catb_E(R)$
in this setting.\footnote{We do not verify this explicitly since it follows directly from the next result
using the conditions $R\in\cata_E(R)$ and $E\in\catb_E(R)$.}
The next result augments this fact significantly.
It is not clear that this result has even been documented in the 
case of semidualizing complexes (that is, the case $\fa=0$). 
\end{ex}

\begin{thm}\label{prop151126a}
Let $M$ be an $\fa$-adic semidualizing $R$-complex,
and consider a set $\{N_i\}_{i\in I}\subseteq\catdb(R)$  such that there are integers $j,t$ such that
$j\leq\inf(N_i)$ and $\sup(N_i)\leq t$ for all $i\in I$.
\begin{enumerate}[\rm(a)]
\item \label{prop151126a1}
One has $N_i\in\cata_M(R)$ for all $i\in I$ if and only if $\prod_iN_i\in\cata_M(R)$.
\item \label{prop151126a2}
One has $N_i\in\catb_M(R)$ for all $i\in I$ if and only if $\bigoplus_iN_i\in\catb_M(R)$.
\end{enumerate}
\end{thm}

\begin{proof}
Note that the conditions $j\leq\inf(N_i)$ and $\sup(N_i)\leq t$ for all $i\in I$ guarantee that
$\prod_iN_i,\bigoplus_iN_i\in\catdb(R)$.

\eqref{prop151126a1}
One implication follows from Proposition~\ref{prop140112a}.
For the converse, assume that we have $N_i\in\cata_M(R)$ for all $i\in I$.
As we have noted, we have $\prod_iN_i\in\catdb(R)$.
We need to show next that $\Lotimes M{\prod_iN_i}\in\catdb(R)$.
Note that the condition $M,\prod_iN_i\in\catdb(R)$ implies $\Lotimes M{\prod_iN_i}\in\catd_+(R)$,
so we need to show that $\Lotimes M{\prod_iN_i}\in\catd_-(R)$.
To this end, we use the following isomorphism from~\cite[Lemma~4.5]{sather:afcc}.
\begin{align*}
\Lotimes K{\Lotimes M{\prod_iN_i}}
&\simeq\prod_i(\Lotimes K{\Lotimes M{N_i}})
\end{align*}
From this, we have the second step in the next display;
the third step is routine.
\begin{align*}
\sup\left(\Lotimes M{\prod_iN_i}\right)
&=\sup\left(\Lotimes K{\Lotimes M{\prod_iN_i}}\right)-n\\
&=\sup\left(\prod_i(\Lotimes K{\Lotimes M{N_i}})\right)-n \\
&=\sup_i\{\sup(\Lotimes K{(\Lotimes M{N_i})})\}-n \\
&=\sup_i\{n+\sup(\Lotimes M{N_i})\}-n \\
&\leq\sup_i\{2n+\sup(M)+\sup(N_i)\} -n\\
&\leq n+\sup(M)+t
\end{align*}
Fact~\ref{cor130528a2} shows that the supports of $\Lotimes M{\prod_iN_i}$ and $\Lotimes M{N_i}$
are contained in $\VE(\fa)$;
so, the first and fourth steps are from~\cite[Lemma~3.1(a)]{sather:afcc}. The fifth step is from
Lemma~\ref{lem151126a}\eqref{lem151126a1}, and the last one is by assumption.

Next, we need to show that the morphism
$$\gamma^M_{\prod_i\! N_i}\colon \prod_iN_i\to\PRhom M{\Lotimes M{\prod_iN_i}}$$ 
is an isomorphism. 
We consider the following commutative diagram in $\catd(R)$.
$$\xymatrix{
\prod_iN_i\ar[r]^-{\gamma^M_{\prod_i\!N_i}}\ar[d]_{\prod_i\gamma^M_{N_i}}^\simeq
&\PRhom M{\Lotimes M{\prod_iN_i}} \ar[d]^\simeq \\
\prod_i\Rhom M{\Lotimes M{N_i}}\ar[r]^-\simeq
&{\PRhom M{\prod_i\Lotimes M{N_i}}} 
}$$ 
The unspecified vertical isomorphism is from~\cite[Theorem~4.7(b)]{sather:afcc},
and the unspecified horizontal one is standard for products.
We conclude that $\gamma^M_{\prod_i\!N_i}$ is an isomorphism, so $\prod_iN_i\in\cata_M(R)$, as desired.

\eqref{prop151126a2}
This is verified similarly, using
Lemma~\ref{lem151126a}\eqref{lem151126a2}
with~\cite[Lemma~3.2(a), Lemma~4.6, and Theorem~4.8(b)]{sather:afcc}.
\end{proof}

\subsection*{Finite Flat Dimension}

We now turn our attention to stability results for Foxby classes with respect to $\Lotimes F-$ where $F$
is an $R$-complex of finite flat dimension. 
Again, it is worth noting the lack of \emph{a priori} boundedness assumptions on $X$ in many of these
(and subsequent) results,
beginning with Theorem~\ref{prop151123ar} from the introduction.
Recall that, given a ring homomorphism $\vf\colon R\to S$, 
we let $\vf^*\colon\spec(S)\to\spec(R)$ denote the induced map.
See~\cite[Proposition~5.6(c)]{sather:afcc}
for perspective on the condition $\vf^*(\supp_S(F))\supseteq\VE(\fa)\bigcap\mspec(R)$.

\begin{thm}\label{prop151123a}
Let $M$ be an $\fa$-adic semidualizing $R$-complex.
Let $F\in\catdb(R)$ be such that $\fd_R(F)<\infty$, and let $X\in\catd(R)$.
If $X\in\catb_M(R)$, then $\Lotimes XF\in\catb_M(R)$ and $\supp_R(X)\subseteq\VE(\fa)$.
The converse of this statement holds when 
at least one of the following conditions is satisfied.
\begin{enumerate}[\rm(1)]
\item\label{prop151123a1}
$F$ is $\fa$-adically finite such that $\supp_R(F)=\VE(\fa)$.
\item\label{prop151123a2}
There is a  homomorphism $\vf\colon R\to S$ of commutative noetherian rings 
with $\fa S\neq S$ such that $F\in\catdb(S)$ is $\fa S$-adically finite over $S$ with $\vf^*(\supp_S(F))\supseteq\VE(\fa)\bigcap\mspec(R)$,
and we have $\Lotimes KX\in\catdf(R)$.
\item\label{prop151123a3}
$F$ is a flat $R$-module with $\supp_R(F)\supseteq\VE(\fa)\bigcap\mspec(R)$, e.g., $F$ is faithfully flat, e.g., free.
\end{enumerate}
\end{thm}

\begin{proof}
Let $\Lotimes{\Rhom MX}F\xra{\omega_{MXF}}\Rhom M{\Lotimes XF}$ be the natural tensor-evaluation morphism.
From Fact~\ref{fact151222a}\eqref{fact151222a1}, 
we know that the induced morphisms $\Lotimes K{\omega_{MXF}}$ and $\Lotimes M{\omega_{MXF}}$
are isomorphisms in $\catd(R)$.
Furthermore, we have the following commutative diagram in $\catd(R)$.
\begin{equation}\label{eq151123a}
\begin{split}
\xymatrix{
\Lotimes M{\Lotimes{\Rhom MX}F}\ar[rd]^>>>>>>>>>>{\Lotimes{\xi^M_X}F}\ar[d]_{\Lotimes M{\omega_{MXF}}}^\simeq
\\
\Lotimes M{\Rhom M{\Lotimes XF}}\ar[r]_-{\xi^M_{\Lotimes XF}}
&\Lotimes XF
}
\end{split}
\end{equation}
Note that if $\Lotimes XF\in\catdb(R)$ and $\supp_R(X)\subseteq\VE(\fa)$ 
and at least one of the conditions~\eqref{prop151123a1}--\eqref{prop151123a3} holds,
then $X\in\catdb(R)$, by~\cite[3.13(b), 4.1(c), and~5.2(c)]{sather:afbha}.
Since we have $\catb_M(R)\subseteq\catdb(R)$, we assume without loss of generality that $X\in\catdb(R)$.

Claim 1: If $\Rhom MX\in\catdb(R)$, then $\Rhom M{\Lotimes XF}\in\catdb(R)$.
Assume that  $\Rhom MX\in\catdb(R)$.
Since Fact~\ref{cor130528a2} implies that we have
$$\cosupp_R(\Rhom M{\Lotimes XF})\subseteq\supp_R(M)\subseteq\VE(\fa)$$
it suffices to show that $\Lotimes K{\Rhom M{\Lotimes XF}}\in\catdb(R)$, by Fact~\ref{lem150604a}.
The next isomorphism from the first paragraph of this proof
$$\Lotimes K{\Rhom M{\Lotimes XF}}\xra[\simeq]{\Lotimes K{\omega_{MXF}}}\Lotimes K{(\Lotimes{\Rhom MX}F)}$$
 implies that it suffices for us to show that 
$\Lotimes K{(\Lotimes{\Rhom MX}F)}\in\catdb(R)$.
The conditions $\Rhom MX\in\catdb(R)$ and $\fd_R(K),\fd_R(F)<\infty$ guarantee that
$\Lotimes K{(\Lotimes{\Rhom MX}F)}\in\catdb(R)$,
so the claim is established.

Claim 2: If $\Rhom M{\Lotimes XF}\in \catdb(R)$ 
and any of the conditions~\eqref{prop151123a1}--\eqref{prop151123a3} hold,
then $\Rhom MX\in\catdb(R)$.
Assume that $\Rhom M{\Lotimes XF}\in \catdb(R)$  
and at least one of the conditions~\eqref{prop151123a1}--\eqref{prop151123a3} holds.
As in the proof of Claim~1, this yields $\Lotimes {(\Lotimes K{\Rhom MX})}F\in\catdb(R)$.

We show how each  of the conditions~\eqref{prop151123a1}--\eqref{prop151123a3}  implies that $\Lotimes K{\Rhom MX}\in\catdb(R)$.
In cases~\eqref{prop151123a1} and~\eqref{prop151123a3}, this is 
from~\cite[Theorem~3.13(b) and Proposition~5.2(c)]{sather:afbha}.
In case~\eqref{prop151123a2}, due to the assumption $X\in\catdb(R)$ from the first paragraph of this proof,
we have $\Lotimes K{\Rhom MX}\in\catdf(R)$ by Lemma~\ref{prop151123z}.
Thus, in this case the  boundedness of $\Lotimes K{\Rhom MX}$ is from~\cite[Theorem~4.1(c)]{sather:afbha}.

Now, the condition $\Lotimes K{\Rhom MX}\in\catdb(R)$ implies that $\Rhom MX\in\catdb(R)$ by Facts~\ref{cor130528a2}
and~\ref{lem150604a}.
This establishes Claim 2.

Claim 3: If $X\in\catb_M(R)$, then $\Lotimes XF\in\catb_M(R)$ and $\supp_R(X)\subseteq\VE(\fa)$.
Assume that $X\in\catb_M(R)$.
Foxby Equivalence~\ref{thm121116a}\eqref{thm121116a2} implies that $\supp_R(X)\subseteq\VE(\fa)$, and Claim~1 implies 
that $\Rhom M{\Lotimes XF}\in\catdb(R)$.
Since $X$ is in $\catb_M(R)$, the morphism $\xi^M_X$ is an isomorphism, hence so are $\Lotimes{\xi^M_X}F$ and 
$\xi^M_{\Lotimes XF}$, because of~\eqref{eq151123a}.
Thus, we have $\Lotimes XF\in\catb_M(R)$, and Claim 3 is established.

We complete the proof by assuming that $\Lotimes XF\in\catb_M(R)$ and $\supp_R(M)\subseteq\VE(\fa)$ and at least one of the 
conditions~\eqref{prop151123a1}--\eqref{prop151123a3} holds, and we prove that $X\in\catb_M(R)$.
Our assumptions imply that $\Rhom M{\Lotimes XF}\in\catdfb(R)$, so Claim 2 implies that $\Rhom MX\in\catdb(R)$. 
Furthermore, the morphism $\xi^M_{\Lotimes XF}$ is an isomorphism, hence so is $\Lotimes{\xi^M_X}F$, because of~\eqref{eq151123a}.

We show how each  of the conditions~\eqref{prop151123a1}--\eqref{prop151123a3} implies that 
$\xi^M_X$ is an isomorphism as well.
Note that the domain and co-domain of $\xi^M_X$ have their supports contained in $\VE(\fa)$, 
one by assumption and the other by Fact~\ref{cor130528a2}. 
Thus, following the logic of Fact~\ref{thm130318aqqc}, we conclude from
Fact~\ref{cor130528a2} and~\cite[Proposition~5.2(b)]{sather:afbha}  that $\xi^M_X$ is an isomorphism in 
cases~\eqref{prop151123a1} and~\eqref{prop151123a3}.
In  case~\eqref{prop151123a2}, we use~\cite[Theorem~4.1(b)]{sather:afbha} similarly;
for this, we need to show that the complexes $\Lotimes KX$ and $\Lotimes K{\Lotimes M{\Rhom MX}}$ are in $\catdf(R)$.
The first of these is from assumption~\eqref{prop151123a2}, and the second one is by Lemma~\ref{prop151123z}, since we have
$\Rhom MX\in\catdb(R)$.
\end{proof}

Our next result is a version of Theorem~\ref{prop151123a} for Auslander classes.

\begin{thm}\label{prop151123b}
Let $M$ be an $\fa$-adic semidualizing $R$-complex.
Let $F\in\catdb(R)$ be such that $\fd_R(F)<\infty$, and let $X\in\catd(R)$.
If $X\in\cata_M(R)$, then $\LL a{\Lotimes XF}\in\cata_M(R)$ and $\cosupp_R(X)\subseteq\VE(\fa)$.
The converse of this statement holds when 
at least one of the conditions~\eqref{prop151123a1}--\eqref{prop151123a3} from Theorem~\ref{prop151123a} holds.
\end{thm}

\begin{proof}
By Fact~\ref{fact151222c} we have
the following isomorphism in $\catd(R)$
\begin{align}\label{eq151222a}
\Lotimes{\Lotimes MX}F
\simeq \Lotimes{M}{\LL a{\Lotimes XF}}.
\end{align}

For the forward implication, assume that $X\in\cata_M(R)$.
Foxby Equivalence~\ref{thm121116a}\eqref{thm121116a3} implies that $\cosupp_R(X)\subseteq\VE(\fa)$ and that $\Lotimes MX\in\catb_M(R)$.
Thus, we have $\Lotimes{M}{\LL a{\Lotimes XF}}\simeq\Lotimes{\Lotimes MX}F\in\catb_M(R)$, by~\eqref{eq151222a} and
Theorem~\ref{prop151123a}.
Fact~\ref{cor130528a2} implies that $\cosupp_R(\LL a{\Lotimes XF})\subseteq\VE(\fa)$,
so another application of Foxby Equivalence~\ref{thm121116a}\eqref{thm121116a3} implies that $\LL a{\Lotimes XF}\in\cata_M(R)$.

The converse is handled similarly, as follows. 
Assume that $\LL a{\Lotimes XF}\in\cata_M(R)$ and $\cosupp_R(X)\subseteq\VE(\fa)$.
Assume also that at least one of the  conditions~\eqref{prop151123a1}--\eqref{prop151123a3} from Theorem~\ref{prop151123a} is satisfied.

Assume for this paragraph that condition~\eqref{prop151123a2} from Theorem~\ref{prop151123a} is satisfied.
Then we have $\supp_R(F)\subseteq\VE(\fa)$ by~\cite[Lemma~5.3]{sather:afcc}, so Fact~\ref{cor130528a2}
implies that $\supp_R(\Lotimes XF)\subseteq\supp_R(F)\subseteq\VE(\fa)$.
The fact that we have $\LL a{\Lotimes XF}\in\catdb(R)$
implies that $\Lotimes XF\in\catdb(R)$ by~\cite[Corollary~3.10(b)]{sather:afbha};
so $X\in\catdb(R)$ by~\cite[Theorem~3.13(b)]{sather:afbha}.
From Lemma~\ref{prop151123z}, it follows that $\Lotimes K{\Lotimes MX}\in\catdf(R)$.

Foxby Equivalence~\ref{thm121116a}\eqref{thm121116a3} and the isomorphism~\eqref{eq151222a} imply that
$\Lotimes{\Lotimes MX}F\simeq\Lotimes{M}{\LL a{\Lotimes XF}}\in\catb_M(R)$.
Since $\supp_R(\Lotimes MX)\subseteq\VE(\fa)$, Theorem~\ref{prop151123a} implies that
$\Lotimes MX\in\catb_M(R)$; in case~\eqref{prop151123a2}, this  uses the condition
$\Lotimes K{\Lotimes MX}\in\catdf(R)$ from the previous paragraph.
Another application of Foxby Equivalence~\ref{thm121116a}\eqref{thm121116a3} implies that $X\in\cata_M(R)$.
\end{proof}

The presence of $\LLno a$ in the previous result may be a bit unsettling.
However, it is a necessary consequence of the co-support condition in Foxby Equivalence~\ref{thm121116a}\eqref{thm121116a3};
we see in the next example that it is unavoidable in general, even over a very nice ring. 
It can be gotten around, though, in the special case $F\in\catdfb(R)$, as we show in the subsequent corollary. 

\begin{ex}\label{ex151126a}
Let $k$ be a field, and set $R:=k[\![Y]\!]$ with $E:=E_R(k)$ and $\fa:=(Y)R$.
Since $R$ is Gorenstein and local, we have $\fd_R(E)<\infty$. 
Also, we have $R\in\cata_E(R)$ by Proposition~\ref{prop130612a}, but $\Lotimes ER\simeq E\notin\cata_E(R)$ by
Example~\ref{ex130530a}.
\end{ex}

\begin{cor}\label{prop151123c}
Let $M$ be an $\fa$-adic semidualizing $R$-complex.
Let $F\in\catdfb(R)$  be such that $\fd_R(F)<\infty$, and let $X\in\catd(R)$.
If $X\in\cata_M(R)$, then $\Lotimes XF\in\cata_M(R)$ and $\cosupp_R(X)\subseteq\VE(\fa)$.
The converse of this statement holds when 
at least one of the following conditions is satisfied.
\begin{enumerate}[\rm(1)]
\item\label{prop151123c1}
$\supp_R(F)=\VE(\fa)$.
\item\label{prop151123c2}
$\supp_R(F)\supseteq\VE(\fa)\bigcap\mspec(R)$
and $\Lotimes KX\in\catdf(R)$.
\item\label{prop151123c3}
$F$ is a flat $R$-module with $\supp_R(F)\supseteq\VE(\fa)\bigcap\mspec(R)$, e.g., $F$ is free.
\end{enumerate}
\end{cor}

\begin{proof}
If $X\in\cata_M(R)$, then $\cosupp_R(X)\subseteq\VE(\fa)$ by Foxby Equivalence~\ref{thm121116a}\eqref{thm121116a3}.
So we assume without loss of generality that $\cosupp_R(X)\subseteq\VE(\fa)$.
Because of this, our assumptions on $F$ imply that $\cosupp_R(\Lotimes XF)\subseteq\cosupp_R(X)\subseteq\VE(\fa)$, 
by Lemma~\ref{lem150612b}.
Fact~\ref{cor130528a3} implies that 
$\Lotimes XF\simeq\LL a{\Lotimes XF}$,
so the desired conclusions follow from Theorem~\ref{prop151123b}.
\end{proof}

\begin{disc}\label{cor151124a}
Let $M$ be an $\fa$-adic semidualizing $R$-complex, and let $X\in\catd(R)$.

The Koszul complex $K$ satisfies condition~\eqref{prop151123a1} of Theorem~\ref{prop151123a} and Corollary~\ref{prop151123c}.
Thus, we have $X\in\catb_M(R)$ if and only if $\Lotimes KX\in\catb_M(R)$ and $\supp_R(X)\subseteq\VE(\fa)$;
and we have $X\in\cata_M(R)$ if and only if $\Lotimes KX\in\cata_M(R)$ and $\cosupp_R(X)\subseteq\VE(\fa)$.

One can similarly use Theorem~\ref{prop151123a} to conclude that
we have $X\in\catb_M(R)$ if and only if $\RG aX\in\catb_M(R)$ and $\supp_R(X)\subseteq\VE(\fa)$;
however this true for more trivial reasons. 
Indeed, If $X\in\catb_M(R)$, then Foxby Equivalence~\ref{thm121116a}\eqref{thm121116a2} implies $\supp_R(X)\subseteq\VE(\fa)$,
so we have $\RG aX\simeq X\in\catb_M(R)$ by Fact~\ref{cor130528a3}.
Conversely, if $\RG aX\in\catb_M(R)$ and $\supp_R(X)\subseteq\VE(\fa)$, then
Fact~\ref{cor130528a3} implies that $X\simeq\RG aX\in\catb_M(R)$.
Similarly, we have $X\in\cata_M(R)$ if and only if $\LL aX\in\cata_M(R)$ and $\cosupp_R(X)\subseteq\VE(\fa)$.
\end{disc}

It is not clear that the converse statements of this section have  been documented in the 
case of semidualizing complexes (that is, the case $\fa=0$). We write this out explicitly for
Theorem~\ref{prop151123a} and leave the remaining cases for the interested reader.

\begin{cor}\label{prop151123az}
Let $C$ be a semidualizing $R$-complex.
Let $F\in\catdb(R)$ be such that $\fd_R(F)<\infty$, and let $X\in\catd(R)$.
If $X\in\catb_C(R)$, then $\Lotimes XF\in\catb_C(R)$.
The converse of this statement holds when 
at least one of the following conditions holds.
\begin{enumerate}[\rm(1)]
\item\label{prop151123az1}
$F\in\catdfb(R)$ satisfies $\supp_R(F)=\spec(R)$.
\item\label{prop151123az2}
There is a  homomorphism $\vf\colon R\to S$ of commutative noetherian rings such that $F\in\catdfb(S)$ 
satisfies $\vf^*(\supp_S(F))\supseteq\mspec(R)$,
and $\Lotimes KX\in\catdf(R)$.
\item\label{prop151123az3}
$F$ is a faithfully flat $R$-module.
\end{enumerate}
\end{cor}

\subsection*{Finite Projective Dimension}
We now turn our attention to stability with respect to $\Rhom P-$ where $P$ is an $R$-complex of finite projective dimension.

\begin{thm}\label{prop151124a}
Let $M$ be an $\fa$-adic semidualizing $R$-complex.
Let $P\in\catdb(R)$ be such that $\pd_R(P)<\infty$, and let $X\in\catd(R)$.
If $X\in\cata_M(R)$, then $\Rhom PX\in\cata_M(R)$ and $\cosupp_R(X)\subseteq\VE(\fa)$.
The converse of this statement holds when 
at least one of the following conditions is satisfied.
\begin{enumerate}[\rm(1)]
\item\label{prop151124a1}
$P$ is $\fa$-adically finite such that $\supp_R(P)=\VE(\fa)$.
\item\label{prop151124a2}
$P$ is $\fa$-adically finite with $\supp_R(P)\supseteq\VE(\fa)\bigcap\mspec(R)$
and $\Lotimes KX\in\catdf(R)$.
\item\label{prop151124a3}
$P$ is a projective $R$-module with $\supp_R(P)\supseteq\VE(\fa)\bigcap\mspec(R)$, e.g., $P$ is faithfully projective, e.g., free.
\end{enumerate}
\end{thm}

\begin{proof}
The proof of this result is very similar to that of Theorem~\ref{prop151123a}, so we only sketch it, highlighting the differences.
If $X\in\cata_M(R)$, then Foxby Equivalence~\ref{thm121116a}\eqref{thm121116a3} implies that $\cosupp_R(X)\subseteq\VE(\fa)$.
Thus, we assume throughout that $\cosupp_R(X)\subseteq\VE(\fa)$.  
From this, if $\Rhom PX\in\catdb(R)$ 
and at least one of the conditions~\eqref{prop151123a1}--\eqref{prop151123a3} holds,
then we have $X\in\catdb(R)$, by~\cite[3.11(b), 4.5(c), and~5.1(e)]{sather:afbha}.
Thus, we assume throughout that $X\in\catdb(R)$.  

The following isomorphism is tensor-evaluation~\cite[Proposition~2.2(vi)]{christensen:apac}. 
$$\Lotimes {\Lotimes KM}{\Rhom PX}\simeq\Rhom P{\Lotimes K{\Lotimes MX}}$$
Thus, if $\Lotimes MX\in\catdb(R)$, then $\Lotimes {\Lotimes KM}{\Rhom PX}\in\catdb(R)$;
and by Fact~\ref{lem150604a}, the condition
$\supp_R(\Lotimes M{\Rhom PX})\subseteq\supp_R(M)\subseteq\VE(\fa)$ implies that we have
$\Lotimes M{\Rhom PX}\in\catdb(R)$. 
Conversely, if $\Lotimes M{\Rhom PX}\in\catdb(R)$ and at least one of the  conditions~\eqref{prop151124a1}--\eqref{prop151124a3} 
is satisfied,
then the complex $\Rhom P{\Lotimes K{\Lotimes MX}}$ is in $\catdb(R)$, 
and it follows from Lemma~\ref{prop151123z} 
and~\cite[3.9(b), 4.3(c), and 5.1(e)]{sather:afbha} that $\Lotimes K{\Lotimes MX}\in\catdb(R)$;
from this, we have $\Lotimes MX\in\catdb(R)$ by Fact~\ref{lem150604a}.
Thus, we assume throughout that $\Lotimes MX\in\catdb(R)$.  

Next, we consider the following commutative diagram in $\catd(R)$
$$\xymatrix{
\Rhom PX\ar[r]^-{\gamma^M_{\Rhom PX}}\ar[d]_{\Rhom P{\gamma^M_X}}
&\Rhom M{\Lotimes M{\Rhom PX}}\ar[d]^{\Rhom M{\omega_{MPX}}}_\simeq \\
\Rhom P{\Rhom M{\Lotimes MX}} \ar[r]^-\simeq
&\Rhom M{\Rhom P{\Lotimes MX}}
}$$
wherein $\omega_{MPX}$ is tensor-evaluation~\cite[Theorem~4.3(b)]{sather:afcc} and the unspecified isomorphism is ``swap'', i.e., 
a composition of adjointness isomorphisms. 
Thus, if $\gamma^M_X$ is an isomorphism, then so is $\gamma^M_{\Rhom PX}$.
Conversely, assume that $\gamma^M_{\Rhom PX}$ is an isomorphism and at least one of the  
conditions~\eqref{prop151124a1}--\eqref{prop151124a3} holds.
It follows that $\Rhom P{\gamma^M_X}$ is an isomorphism.
Since the source and target of $\gamma^M_X$ both have their co-supports contained in $\VE(\fa)$,
Fact~\ref{cor130528a2} and~\cite[Proposition~5.1(d)]{sather:afbha} show that $\gamma^M_X$ is an isomorphism in 
cases~\eqref{prop151124a1} and~\eqref{prop151124a3}, respectively; see Fact~\ref{thm130318aqqc}.
In  case~\eqref{prop151124a2}, we use~\cite[Theorem~4.3(b)]{sather:afbha} similarly;
for this, we need to show that $\Lotimes KX,\Lotimes K{\Rhom M{\Lotimes MX}}\in\catdf(R)$.
The first of these is from assumption~\eqref{prop151123a2}, and the second one is by Lemma~\ref{prop151123z}, as we have
$\Lotimes MX\in\catdb(R)$.
\end{proof}

\begin{thm}\label{prop151124b}
Let $M$ be an $\fa$-adic semidualizing $R$-complex.
Let $X\in\catd(R)$ and  $P\in\catdb(R)$ be such that $\pd_R(P)<\infty$.
If $X\in\catb_M(R)$, then $\RG a{\Rhom PX}\in\catb_M(R)$ and $\supp_R(X)\subseteq\VE(\fa)$.
The converse of this statement holds when 
at least one of the conditions~\eqref{prop151124a1}--\eqref{prop151124a3} from Theorem~\ref{prop151124a} holds.
\end{thm}

\begin{proof}
Again, we sketch the proof. 
By Foxby Equivalence~\ref{thm121116a}\eqref{thm121116a2}, we assume without loss of generality that $\supp_R(X)\subseteq\VE(\fa)$.
From Fact~\ref{fact151222c}, we have the first isomorphism in the next sequence
\begin{align*}
\Rhom M{\RG a{\Rhom PX}}
&\simeq\Rhom M{\Rhom PX}\\
&\simeq\Rhom P{\Rhom MX}.
\end{align*}
The second isomorphism is swap.

Now, for the forward implication, assume that  $X\in\catb_M(R)$.
Then we have 
$\Rhom MX\in\cata_M(R)$, by Foxby Equivalence~\ref{thm121116a}\eqref{thm121116a2}. 
Theorem~\ref{prop151124a} implies that $\Rhom P{\Rhom MX}\in\cata_M(X)$.
From the above isomorphisms, it follows that we have $\Rhom M{\RG a{\Rhom PX}}\in\cata_M(X)$.
Fact~\ref{cor130528a2} implies that we also have $\supp_R(\RG a{\Rhom PX})\subseteq\VE(\fa)$,
so we conclude that $\RG a{\Rhom PX}\in\catb_M(R)$ by Foxby Equivalence~\ref{thm121116a}\eqref{thm121116a2}.
This completes the proof of the forward implication.

Assume for this paragraph that $\RG a{\Rhom PX}\in\catdb(R)$ and
condition~\eqref{prop151124a2} from Theorem~\ref{prop151124a} is satisfied.
Fact~\ref{cor130528a2} implies that
$$\cosupp_R(\Rhom PX)\subseteq\supp_R(P)\subseteq\VE(\fa)$$
so we have  $\Rhom PX\in\catdb(R)$ by~\cite[Corollary~3.15(b)]{sather:afbha};
we conclude that $X\in\catdb(R)$ by~\cite[Theorem~4.5(c)]{sather:afbha}.
In particular, we have $\Lotimes K{\Rhom MX}\in\catdf(R)$ by Lemma~\ref{prop151123z}.

Now, for the converse, assume that $\RG a{\Rhom PX}\in\catb_M(R)$ and at least one of the 
conditions~\eqref{prop151124a1}--\eqref{prop151124a3} from Theorem~\ref{prop151124a} holds.
Foxby Equivalence~\ref{thm121116a}\eqref{thm121116a2} and the isomorphisms above imply that
$$\Rhom P{\Rhom MX}\simeq\Rhom M{\RG a{\Rhom PX}}\in\cata_M(X).$$
Since $\cosupp_R(\Rhom MX)\subseteq\VE(\fa)$ by Fact~\ref{cor130528a2}, we have
$\Rhom MX\in\cata_M(R)$ by Theorem~\ref{prop151124a}; 
in case~\eqref{prop151123a2}, this uses the condition $\Lotimes K{\Rhom MX}\in\catdf(R)$
from the previous paragraph.
A final application of Foxby Equivalence~\ref{thm121116a}\eqref{thm121116a2} implies that $X\in\catb_M(R)$.
\end{proof}

The next example shows that one cannot drop the $\RGno a$ from Theorem~\ref{prop151124b},
even when $P$ is free.
See, however, Corollary~\ref{prop151124c} for the special case $P\in\catdfb(R)$.

\begin{ex}\label{ex151126b}
Let $k$ be a field, and set $R:=k[\![Y]\!]$ with $E:=E_R(k)$.
Then we have $\Rhom{R^{(\bbn)}}E\simeq E^{\bbn}\notin\catb_E(R)$ and $E\in\catb_E(R)$, by Example~\ref{ex151125a}.
\end{ex}

\begin{cor}\label{prop151124c}
Let $M$ be an $\fa$-adic semidualizing $R$-complex.
Let $P\in\catdfb(R)$  be such that $\pd_R(P)<\infty$, and let $X\in\catd(R)$.
If $X\in\catb_M(R)$, then one has $\Rhom PX\in\catb_M(R)$ and $\supp_R(X)\subseteq\VE(\fa)$.
The converse of this statement holds when 
at least one of the following conditions is satisfied.
\begin{enumerate}[\rm(1)]
\item\label{prop151124c1}
$\supp_R(P)=\VE(\fa)$.
\item\label{prop151124c2}
$\supp_R(P)\supseteq\VE(\fa)\bigcap\mspec(R)$,
and $\Lotimes KX\in\catdf(R)$.
\item\label{prop151124c3}
$P$ is a projective $R$-module with $\supp_R(P)\supseteq\VE(\fa)\bigcap\mspec(R)$, e.g., $P$ is free.
\end{enumerate}
\end{cor}

\begin{proof}
Again, assume without loss of generality that $\supp_R(X)\subseteq\VE(\fa)$.
Because of this, our assumptions on $P$ imply that $\supp_R(\Rhom PX)\subseteq\supp_R(X)\subseteq\VE(\fa)$, 
by Lemma~\ref{lem150612b}.
Fact~\ref{cor130528a3} implies that $\Rhom PX\simeq\RG a{\Rhom PX}$.
Thus, the desired conclusions follow from Theorem~\ref{prop151124b}.
\end{proof}

\subsection*{Finite Injective Dimension}

The next three results are verified like earlier ones.\footnote{In 
Corollary~\ref{prop151124ac}, use~\cite[Proposition~3.16]{sather:scc} to conclude that $\supp_R(\Rhom XI)\subseteq\VE(\fa)$.}
For perspective in these results, we recall the following.

\begin{disc}\label{disc151222a}
Let $J$ be an injective $R$-module, with Matlis decomposition $J\cong\bigoplus_{\p\in\spec(R)}E_R(R/\p)^{(\mu_\p)}$.
Then one has
\begin{align*}
\supp_R(J)&=\{\p\in\spec(R)\mid\mu_\p\neq\emptyset\} \\
\cosupp_R(J)&=\{\q\in\spec(R)\mid\text{there is a $\p\in\spec(R)$ such that $\q\subseteq\p$ and $\mu_\p\neq\emptyset$}\}
\end{align*}
by~\cite[Propositions~3.8 and~6.3]{sather:scc}. 
From this, one verifies readily that 
\begin{enumerate}[\rm(a)]
\item $\supp_R(J)\subseteq\cosupp_R(J)$, and
\item $\VE(\fa)\subseteq\cosupp_R(J)$ if and only if $\supp_R(J)\supseteq\VE(\fa)\bigcap\mspec(R)$.
\end{enumerate}
\end{disc}

\begin{thm}\label{prop151124aa}
Let $M$ be an $\fa$-adic semidualizing $R$-complex.
Let $I\in\catdb(R)$ be such that $\id_R(I)<\infty$, and let $X\in\catd(R)$.
If $X\in\catb_M(R)$, then $\Rhom XI\in\cata_M(R)$ and $\supp_R(X)\subseteq\VE(\fa)$.
The converse of this statement holds when 
$I$ is an injective $R$-module with $\cosupp_R(I)\supseteq\VE(\fa)$, e.g., $I$ is faithfully injective.
\end{thm}

\begin{thm}\label{prop151124ab}
Let $M$ be an $\fa$-adic semidualizing $R$-complex.
Let $X\in\catd(R)$, and let $I\in\catdb(R)$ be such that $\id_R(I)<\infty$.
If $X\in\cata_M(R)$, then $\RG a{\Rhom XI}\in\catb_M(R)$ and $\cosupp_R(X)\subseteq\VE(\fa)$.
The converse of this statement holds when
$I$ is an injective $R$-module with $\cosupp_R(I)\supseteq\VE(\fa)$, e.g., $I$ is faithfully injective.
\end{thm}

The next example shows that one cannot avoid the $\RGno a$ in Theorem~\ref{prop151124ab}, even in very nice situations. 

\begin{ex}\label{ex151126c}
Let $k$ be a field, and set $R:=k[\![Y]\!]$ with $E:=E_R(k)$
and  $I:=E_R(R)$.  Remark~\ref{disc151222a} shows that we have
$\supp_R(I)=\{0\}\not\subseteq\VE(\fa)$.
Then Foxby Equivalence~\ref{thm121116a}\eqref{thm121116a2} implies that
$\Rhom RI\simeq I\notin\catb_E(R)$,
even though by Proposition~\ref{prop130612a} implies that we have $R\in\cata_E(R)$.
\end{ex}

\begin{cor}\label{prop151124ac}
Let $M$ be an $\fa$-adic semidualizing $R$-complex.
Let $X\in\catdf(R)$ and $I\in\catdb(R)$ be such that $\id_R(I)<\infty$ and $\supp_R(I)\subseteq\VE(\fa)$.
If $X\in\cata_M(R)$, then $\Rhom XI\in\catb_M(R)$ and $\cosupp_R(X)\subseteq\VE(\fa)$.
The converse of this statement holds when 
$I$ is an injective $R$-module with $\cosupp_R(I)\supseteq\VE(\fa)$.
\end{cor}

\begin{disc}\label{disc151126a}
One has to be a bit careful with the converse  in Corollary~\ref{prop151124ac} to make sure that
one satisfies both assumptions $\supp_R(I)\subseteq\VE(\fa)\subseteq\cosupp_R(I)$. For instance, this 
will fail in general when $I$ is faithfully injective,
by Remark~\ref{disc151222a}.
\end{disc}

\begin{disc}\label{disc151213a}
One can use the results in this section in a variety of combinations. 
For instance, combining Theorems~\ref{prop151123a} and~\ref{prop151124aa}, one obtains the following.

Let $M$ be an $\fa$-adic semidualizing $R$-complex.
Let $F\in\catdb(R)$ be such that $\fd_R(F)<\infty$, and let $X\in\catd(R)$.
Let $I$ be an injective $R$-module with $\cosupp_R(I)\supseteq\VE(\fa)$, e.g., $I$ is faithfully injective,
and set $J=\Rhom FI$. 
If $X\in\catb_M(R)$, then $\Rhom XI\in\cata_M(R)$ and $\supp_R(X)\subseteq\VE(\fa)$.
The converse holds if 
at least one of the  conditions~\eqref{prop151123a1}--\eqref{prop151123a3} from Theorem~\ref{prop151123a} holds.

Indeed, one has $\id_R(J)<\infty$, so the forward implication
follows directly from Theorem~\ref{prop151124aa}
Also, by definition and Hom-tensor adjointness, we have
\begin{align*}
\Rhom XJ
&\simeq\Rhom X{\Rhom FI}
\simeq\Rhom{\Lotimes XF}I
\end{align*}
so the converse follows by applying first Theorem~\ref{prop151124aa}
and then Theorem~\ref{prop151123a}.
We leave other variations on this theme to the interested reader.
\end{disc}

\section{Base Change}\label{sec150527a}

This  section focuses on some transfer properties for  Foxby classes.
It contains Theorem~\ref{prop151127ar} from the introduction. 

\begin{notation}\label{notn160110a}
In this section, let $\vf\colon R\to S$ be a  homomorphism of commutative noetherian rings with
$\fa S\neq S$, and let $Q\colon \catd(S)\to\catd(R)$ be the forgetful functor.
\end{notation}

If $C$ is a semidualizing $R$-complex, then~\cite[Theorem 5.1]{christensen:scatac} says that
$S\in\catac(R)$ if and only if $\Lotimes SC$ is a semidualizing $S$-complex.
Proposition~\ref{prop140218a} shows that things are not so simple for adic semidualizing complexes.
Specifically, assume that $R$ is not $\fa$-adically complete, and let $M$ be an $\fa$-adic semidualizing $R$-complex.
Then $M\simeq\Lotimes RM$ is $\fa$-adically semidualizing over $R$, but $R\notin\cata_M(R)$ by
Proposition~\ref{prop140218a}.
As one might expect, the missing ingredient involves co-support, embodied in the completeness condition in our next result.

\begin{thm}\label{prop150525a}
Let $M$ be  $\fa$-adic semidualizing over $R$.
Then  $S\in\cata_M(R)$ if and only if
$\Lotimes SM$ is $\fa S$-adically semidualizing over $S$ and $S$ is $\fa S$-adically complete.
\end{thm}

\begin{proof}
For this paragraph, assume that $S$ is $\fa S$-adically complete and $\Lotimes SM\in\catdb(S)$. 
Since $M$ is $\fa$-adically semidualizing, we have $\supp_R(M)\subseteq\VE(\fa)$, and hence
$\supp_S(\Lotimes SM)\subseteq\VE(\fa S)$ by~\cite[Lemma~5.7]{sather:afcc}.
Thus, the morphism $\chi^S_{\Lotimes SM}=\chi^{\comp S^{\fa S}}_{\Lotimes SM}$ is defined, and
we have the following commutative diagram in $\catd(S)$:
\begin{equation}\label{eq150525a}
\begin{split}
\xymatrix{S\ar[rd]^{\chi^S_{\Lotimes SM}}\ar[d]_-{\gamma^M_S}
\\
\Rhom{M}{\Lotimes SM}\ar[r]_-\simeq
&\Rhom[S]{\Lotimes SM}{\Lotimes SM}.}
\end{split}
\end{equation}
The unspecified isomorphism is from Hom-tensor adjointness.

Now, for the forward implication, assume that $S\in\cata_M(R)$.
Foxby Equivalence~\ref{thm121116a}\eqref{thm121116a3} implies that
$\cosupp_R(S)\subseteq\VE(\fa)$, so we conclude that $\cosupp_S(S)\subseteq\VE(\fa S)$ by~\cite[Lemma~5.4]{sather:afcc}.
It follows from Fact~\ref{cor130528a3} that the natural morphism $S\to \mathbf{L}\Lambda^{\fa S}(S)$ is an isomorphism in $\catd(S)$,
i.e., $S$ is $\fa S$-adically complete.
Furthermore, the condition $S\in\cata_M(R)$ implies by definition that $\Lotimes SM\in\catdb(R)$ and $\gamma^M_S$ is an isomorphism.
In particular, we are in the situation of the first paragraph of this proof, and
the diagram~\eqref{eq150525a} implies that $\chi^S_{\Lotimes SM}$ is an isomorphism.
Since $M$ is $\fa$-adically finite over $R$, \cite[Theorem~5.10]{sather:afcc} implies that $\Lotimes SM$ is $\fa S$-adically finite over $S$,
so $\Lotimes SM$ is $\fa S$-adically semidualizing over $S$, as desired.

For the converse, assume that 
$\Lotimes SM$ is $\fa S$-adically semidualizing over $S$ and $S$ is $\fa S$-adically complete.
In particular, we are in the situation of the first paragraph of this proof,
and the morphism $\chi^S_{\Lotimes SM}$ is an isomorphism in $\catd(S)$.
The diagram~\eqref{eq150525a} implies that $\gamma^M_S$ is an isomorphism,
so $S\in\cata_M(R)$, as desired.
\end{proof}

Theorem~\ref{prop150525a} gives some perspective on the $\fa S$-adically semidualizing condition for $\Lotimes SM$
in the next few results, as does~\cite[Theorem~5.6]{sather:asc}, 
which says that $\Lotimes SM$ is $\fa S$-adically semidualizing over $S$ whenever $\fd_R(S)<\infty$.

\begin{prop}\label{prop140218b}
Let $M$ be an $\fa$-adic semidualizing $R$-complex.
Assume that $\Lotimes SM$ is $\fa S$-adically semidualizing over $S$, and let $Y$ be an $S$-complex. 
\begin{enumerate}[\rm(a)]
\item \label{prop140218b1}
One has $Y\in\cata_{\Lotimes SM}(S)$ if and only if $Q(Y)\in\cata_{M}(R)$.
\item \label{prop140218b2}
One has $Y\in\catb_{\Lotimes SM}(S)$ if and only if $Q(Y)\in\catb_{M}(R)$.
\end{enumerate}
\end{prop}

\begin{proof}
Argue as in the proof of~\cite[Proposition 5.3]{christensen:scatac}.
\end{proof}

\begin{disc}\label{disc151213b}
For perspective in some of our subsequent results, note that~\cite[Proposition~5.6(a)]{sather:afcc} implies that $\supp_R(S)=\vf^*(\spec(S))$.

Also, we have $\Lotimes KS\in\catdf(R)$ if and only if the induced map $R/\fa\to S/\fa S$
is module-finite. (For instance, this is satisfied when $S$ is module-finite over $R$
or when $S=\Comp Rb$ for some ideal $\fb\subseteq\fa$.)
Indeed, we have $\HH_0(\Lotimes KS)\cong S/\fa S$; thus, if $\Lotimes KS\in\catdf(R)$,
then $S/\fa S$ is finitely generated over $R$, hence over $R/\fa R$. 
For the converse, note that each module $\HH_i(\Lotimes KS)$ is finitely generated over $S$, hence over $S/\fa S$;
thus, if the induced map $R/\fa\to S/\fa S$
is module-finite these homology modules are finitely generated over $R/\fa$, hence over $R$.

It follows that $S$ is $\fa$-adically finite over $R$ if and only if $\supp_R(S)\subseteq\VE(\fa)$ and the induced map $R/\fa\to S/\fa S$
is module-finite. 
\end{disc}

\subsection*{Base Change for Bass Classes}
Here is Theorem~\ref{prop151127ar} from the introduction.

\begin{thm}\label{prop151127a}
Let $M$ be an $\fa$-adic semidualizing $R$-complex, and
assume that $\fd_R(S)<\infty$.
Let $X\in\catd(R)$ be given, and consider the following conditions.
\begin{enumerate}[\rm(i)]
\item \label{prop151127a1}
$X\in\catb_M(R)$.
\item \label{prop151127a2}
$\Lotimes SX\in\catb_{M}(R)$ and $\supp_R(X)\subseteq\VE(\fa)$.
\item \label{prop151127a3}
$\Lotimes SX\in\catb_{\Lotimes SM}(S)$ and $\supp_R(X)\subseteq\VE(\fa)$.
\end{enumerate}
Then we have~\eqref{prop151127a1}$\implies$\eqref{prop151127a2}$\iff$\eqref{prop151127a3}.
The conditions~\eqref{prop151127a1}--\eqref{prop151127a3} are equivalent when 
at least one of the following conditions is satisfied.
\begin{enumerate}[\rm(1)]
\item\label{prop151127a4}
$S$ is $\fa$-adically finite over $R$ such that $\supp_R(S)=\VE(\fa)$.
\item\label{prop151127a5}
$\supp_R(S)\supseteq\VE(\fa)\bigcap\mspec(R)$,
and $\Lotimes KX\in\catdf(R)$.
\item\label{prop151127a6}
$S$ is  flat over $R$ with $\supp_R(S)\supseteq\VE(\fa)\bigcap\mspec(R)$, e.g., $S$ is faithfully flat.
\end{enumerate}
\end{thm}

\begin{proof}
The implication~\eqref{prop151127a1}$\implies$\eqref{prop151127a2} and its conditional converse follow from
Theorem~\ref{prop151123a}. 
The equivalence~\eqref{prop151127a2}$\iff$\eqref{prop151127a3}
is from Proposition~\ref{prop140218b}\eqref{prop140218b1}. 
\end{proof}

The point of the next result is to shift the support conditions in Theorem~\ref{prop151127a}.

\begin{cor}\label{prop151213a}
Let $M$ be an $\fa$-adic semidualizing $R$-complex, and
assume that $\fd_R(S)<\infty$.
Let $X\in\catd(R)$ be such that $\supp_R(X)\subseteq\supp_R(S)$.
Consider the following conditions.
\begin{enumerate}[\rm(i)]
\item \label{prop151213a1}
$X\in\catb_M(R)$.
\item \label{prop151213a2}
$\Lotimes SX\in\catb_{M}(R)$.
\item \label{prop151213a3}
$\Lotimes SX\in\catb_{\Lotimes SM}(S)$.
\end{enumerate}
Then we have~\eqref{prop151213a1}$\implies$\eqref{prop151213a2}$\iff$\eqref{prop151213a3}.
The conditions~\eqref{prop151213a1}--\eqref{prop151213a3} are equivalent when 
at least one of the following conditions is satisfied.
\begin{enumerate}[\rm(1)]
\item\label{prop151213a4}
$S$ is $\fa$-adically finite such that $\supp_R(S)=\VE(\fa)$.
\item\label{prop151213a5}
$\supp_R(S)\supseteq\VE(\fa)\bigcap\mspec(R)$,
and $\Lotimes KX\in\catdf(R)$.
\item\label{prop151213a6}
$S$ is  flat over $R$ with $\supp_R(S)\supseteq\VE(\fa)\bigcap\mspec(R)$, e.g., $S$ is faithfully flat.
\end{enumerate}
\end{cor}

\begin{proof}
Conditions~\eqref{prop151213a2} and~\eqref{prop151213a3} are equivalent by Proposition~\ref{prop140218b}\eqref{prop140218b1},
and the implication~\eqref{prop151213a1}$\implies$\eqref{prop151213a2} is from Theorem~\ref{prop151123a}. 
By Theorem~\ref{prop151127a}, it remains to assume that $\Lotimes SX\in\catb_{M}(R)$,
and show that $\supp_R(X)\subseteq\VE(\fa)$.
The assumption $\supp_R(X)\subseteq\supp_R(S)$ explains the first step in the next display,
and the second step is from Fact~\ref{cor130528a2}.
\begin{align*}
\supp_R(X)
&=\supp_R(X)\bigcap\supp_R(S)
=\supp_R(\Lotimes SX)
\subseteq\VE(\fa)
\end{align*}
The last step here is from Foxby Equivalence~\ref{thm121116a}\eqref{thm121116a2}, since $\Lotimes SX\in\catb_{M}(R)$.
\end{proof}

\begin{cor}\label{prop151127b}
Let $M$ be an $\fa$-adic semidualizing $R$-complex.
Given an $R$-complex $X\in\catd(R)$,  the following conditions are equivalent.
\begin{enumerate}[\rm(i)]
\item \label{prop151127b1}
$X\in\catb_M(R)$.
\item \label{prop151127b2}
$\Lotimes {\Comp Ra}X\in\catb_{M}(R)$ and $\supp_R(X)\subseteq\VE(\fa)$.
\item \label{prop151127b3}
$\Lotimes {\Comp Ra}X\in\catb_{\Lotimes {\Comp Ra}M}({\Comp Ra})$ and $\supp_R(X)\subseteq\VE(\fa)$.
\end{enumerate}
\end{cor}

\begin{proof}
The completion $\Comp Ra$ 
is  flat over $R$ with $\supp_R(\Comp Ra)\supseteq\VE(\fa)\bigcap\mspec(R)$.
Thus, the desired result follows from Theorem~\ref{prop151127a}, using condition~\eqref{prop151127a6}.
\end{proof}

\subsection*{Base Change for Auslander Classes}

\begin{thm}\label{prop151127c}
Let $M$ be an $\fa$-adic semidualizing $R$-complex, and
assume that $\fd_R(S)<\infty$.
Let $X\in\catd(R)$ be given, and consider the following conditions.
\begin{enumerate}[\rm(i)]
\item \label{prop151127c1}
$X\in\cata_M(R)$.
\item \label{prop151127c2}
$\LL a{\Lotimes SX}\in\cata_{M}(R)$ and $\cosupp_R(X)\subseteq\VE(\fa)$.
\item \label{prop151127c3}
$\LLS a{\Lotimes SX}\in\cata_{\Lotimes SM}(S)$ and $\cosupp_R(X)\subseteq\VE(\fa)$.
\end{enumerate}
Then we have~\eqref{prop151127c1}$\implies$\eqref{prop151127c2}$\iff$\eqref{prop151127c3}.
The conditions~\eqref{prop151127c1}--\eqref{prop151127c3} are equivalent when 
at least one of the following conditions is satisfied.
\begin{enumerate}[\rm(1)]
\item\label{prop151127c4}
$S$ is $\fa$-adically finite over $R$ such that $\supp_R(S)=\VE(\fa)$.
\item\label{prop151127c5}
$\supp_R(S)\supseteq\VE(\fa)\bigcap\mspec(R)$,
and $\Lotimes KX\in\catdf(R)$.
\item\label{prop151127c6}
$S$ is  flat over $R$ with $\supp_R(S)\supseteq\VE(\fa)\bigcap\mspec(R)$, e.g., $S$ is faithfully flat.
\end{enumerate}
\end{thm}

\begin{proof}
By~\cite[Lemma~5.2]{sather:afcc}, we have $Q(\LLS a{\Lotimes SX})\simeq\LL a{\Lotimes SX}$ in $\catd(R)$.
Thus, one verifies the desired conclusions as in the proof of Theorem~\ref{prop151127a},
using Theorem~\ref{prop151123b} and Proposition~\ref{prop140218b}\eqref{prop140218b1}.
\end{proof}

One might expect a version of Corollary~\ref{prop151213a} to follow here. 
The key point of the proof of such a result would be to  assume that $\cosupp_R(X)\subseteq\supp_R(S)$ and
$\LL a{\Lotimes SX}\in\cata_{M}(R)$,
and then show that $\cosupp_R(X)\subseteq\VE(\fa)$. 
However, the next example shows that this implication fails in general.

\begin{ex}\label{ex151214a}
Let $k$ be a field, and consider the localized polynomial ring $R=k[Y]_{Yk[Y]}$.
Set $\fa:=YR$ and $E:=E_R(k)$.
Since $\Comp Ra$ is faithfully flat over $R$ and $R$ is not $\fa$-adically complete, we have 
$$\cosupp_R(R)=\spec(R)=\supp_R(\Comp Ra)\not\subseteq\VE(\fa)$$
by~\cite[Proposition~6.10]{sather:scc}. On the other hand, 
we have
$$\LL a{\Lotimes{\Comp Ra}R}\simeq\LL a{\Comp Ra}\simeq\Comp Ra\in\cata_{E}(R)$$
by Proposition~\ref{prop130612a}.
\end{ex}

The next result is proved like Corollary~\ref{prop151127b}, using Theorem~\ref{prop151127c}.

\begin{cor}\label{prop151127d}
Let $M$ be an $\fa$-adic semidualizing $R$-complex, and
let $X\in\catd(R)$ be given. Then the following conditions are equivalent.
\begin{enumerate}[\rm(i)]
\item \label{prop151127d1}
$X\in\cata_M(R)$.
\item \label{prop151127d2}
$\LL a{\Lotimes {\Comp Ra}X}\in\cata_{M}(R)$ and $\cosupp_R(X)\subseteq\VE(\fa)$.
\item \label{prop151127d3}
$\LLa a{\Lotimes {\Comp Ra}X}\in\cata_{\Lotimes {\Comp Ra}M}({\Comp Ra})$ and $\cosupp_R(X)\subseteq\VE(\fa)$.
\end{enumerate}
\end{cor}

The next two results show how to remove the derived local homology from the previous two results, 
in the presence of extra finiteness conditions. 

\begin{cor}\label{prop151127e}
Let $M$ be an $\fa$-adic semidualizing $R$-complex, and
assume that $S$ is module-finite over $R$ with $\fd_R(S)<\infty$.
Let $X\in\catd(R)$ be given, and consider the following conditions.
\begin{enumerate}[\rm(i)]
\item \label{prop151127e1}
$X\in\cata_M(R)$.
\item \label{prop151127e2}
${\Lotimes SX}\in\cata_{M}(R)$ and $\cosupp_R(X)\subseteq\VE(\fa)$.
\item \label{prop151127e3}
${\Lotimes SX}\in\cata_{\Lotimes SM}(S)$ and $\cosupp_R(X)\subseteq\VE(\fa)$.
\end{enumerate}
Then we have~\eqref{prop151127e1}$\implies$\eqref{prop151127e2}$\iff$\eqref{prop151127e3}.
The conditions~\eqref{prop151127e1}--\eqref{prop151127e3} are equivalent when 
at least one of the following conditions is satisfied.
\begin{enumerate}[\rm(1)]
\item\label{prop151127e4}
$\supp_R(S)=\VE(\fa)$.
\item\label{prop151127e5}
$\supp_R(S)\supseteq\VE(\fa)\bigcap\mspec(R)$,
and $\Lotimes KX\in\catdf(R)$.
\item\label{prop151127e6}
$S$ is  flat over $R$ with $\supp_R(S)\supseteq\VE(\fa)\bigcap\mspec(R)$, e.g., $S$ is faithfully flat.
\end{enumerate}
\end{cor}

\begin{proof}
The module-finite assumption on $S$ says that $S\in\catdfb(R)$. 
Thus, the desired conclusions follow from Corollary~\ref{prop151123c} and Proposition~\ref{prop140218b}\eqref{prop140218b1}.
\end{proof}

\begin{cor}\label{prop151127f}
Let $M$ be an $\fa$-adic semidualizing $R$-complex.
Given an $R$-complex $X\in\catdf(R)$,  the following conditions are equivalent.
\begin{enumerate}[\rm(i)]
\item \label{prop151127f1}
$X\in\cata_M(R)$.
\item \label{prop151127f2}
${\Lotimes {\Comp Ra}X}\in\cata_{M}(R)$ and $\cosupp_R(X)\subseteq\VE(\fa)$.
\item \label{prop151127f3}
${\Lotimes {\Comp Ra}X}\in\cata_{\Lotimes {\Comp Ra}M}({\Comp Ra})$ and $\cosupp_R(X)\subseteq\VE(\fa)$.
\end{enumerate}
\end{cor}

\begin{proof}
Assume without loss of generality that $\cosupp_R(X)\subseteq\VE(\fa)$.
Then~\cite[Theorem~4.2(c)]{sather:afbha} implies that $X\in\catdb(R)$ if and only if $\Lotimes{\Comp Ra}X\in\catdb(R)$.
Thus, we assume without loss of generality that $X\in\catdb(R)$, that is, $X\in\catdfb(R)$.
It follows that $\Lotimes{\Comp Ra}X\in\catdfb(\Comp Ra)$, so each homology module $\HH_i(\Lotimes{\Comp Ra}X)$
is $\fa$-adically complete over $R$ and $\fa\Comp Ra$-adically complete over $\Comp Ra$.
Fact~\ref{cor130528a3} thus implies that
$\LL a{\Lotimes {\Comp Ra}X}\simeq
\Lotimes{\Comp Ra}X$
in $\catd(R)$, and
$\LLa a{\Lotimes {\Comp Ra}X}\simeq
\Lotimes{\Comp Ra}X$
in $\catd(\Comp Ra)$.
Thus, the desired conclusions follow from Corollary~\ref{prop151127d}.
\end{proof}

\subsection*{Local-Global Behavior}
To keep the notation under control in the next few results, 
we write $U^{-1}X$ for $\Lotimes{(U^{-1}R)} X$, and similarly for $U^{-1}M$, $X_\p$, etc.
Note that in each result, each localization of $M$ is appropriately adically semidualizing
over the localized ring by~\cite[Theorem~5.7]{sather:asc}; this is why we restrict to localizations that are well-behaved with
respect to $\fa$. In turn, this is why we need to assume that $\supp_R(X)\subseteq \VE(\fa)$: for instance, 
if $\n\in\mspec(R)\ssm\VE(\fa)$, then $X=R/\n$ satisfies condition~\eqref{thm151214a4} in the theorem, but not condition~\eqref{thm151214a1}. 

\begin{thm}\label{thm151214a}
Let $M$ be an $\fa$-adic semidualizing $R$-complex, and let $X\in\catdb(R)$ be such that $\supp_R(X)\subseteq\VE(\fa)$.
Then the following conditions are equivalent.
\begin{enumerate}[\rm(i)]
\item\label{thm151214a1}
$X\in\catb_M(R)$.
\item\label{thm151214a2}
for each multiplicatively closed subset $U\subseteq R$ such that $U^{-1}\fa \neq U^{-1}R$, 
we have $U^{-1}X\in \catb_{U^{-1}M}(U^{-1}R)$.
\item\label{thm151214a3}
For all $\p\in\VE(\fa)$, we have $X_\p\in\catb_{M_\p}(R_\p)$.
\item\label{thm151214a4}
For all $\m\in\VE(\fa)\bigcap\mspec(R)$, we have $X_\m\in\catb_{M_\m}(R_\m)$.
\end{enumerate}
\end{thm}

\begin{proof}
In light of Theorem~\ref{prop151127a}, it suffices to prove the implication~\eqref{thm151214a4}$\implies$\eqref{thm151214a1}.
Assume that for all $\m\in\VE(\fa)\bigcap\mspec(R)$, we have $X_\m\in\catb_{M_\m}(R_\m)$.
From Proposition~\ref{prop140218b}\eqref{prop140218b2} we have
$X_\m\in\catb_{M}(R)$ for all such $\m$.
Note that we have
$$-\infty<\inf(X)\leq\inf(X_\m)\qquad\text{and}\qquad\sup(X_\m)\leq\sup(X)<\infty$$
for all $\m$.
Thus, Theorem~\ref{prop151126a}\eqref{prop151126a2} implies that
$$\Lotimes X{\left(\bigoplus_{\m}R_\m\right)}\simeq
\bigoplus_{\m}X_\m\in\catb_{M}(R)$$
where the sums are taken over all $\m\in\VE(\fa)\bigcap\mspec(R)$.
Since the module $\bigoplus_{\m}R_\m$ is  flat over $R$ with support containing $\VE(\fa)\bigcap\mspec(R)$,
it follows from 
Theorem~\ref{prop151123a} that we have $X\in\catb_M(R)$, as desired.
\end{proof}

The next example shows why we need $X\in\catdb(R)$ in the previous result.

\begin{ex}\label{ex151214c}
Let $k$ be a field and consider the polynomial ring $R:=k[X]$ with $\fa=0$.
Let $\{\m_i\}_{\i\in\bbz}$ be a set of distinct maximal ideals of $R$, and set $X:=\bigoplus_{i\in\bbz}\shift^iR/\m_i$. 
Then for each $\m\in\mspec(R)$, we have $X_\m\simeq\shift^i\kappa(\m)$ if $\m=\m_i$ for some $i$,
and $X_\m\simeq 0$ otherwise.
In particular, this implies $X_\m\in\catdb(R_\m)=\catb_{R_\m}(R_\m)$ for all $\m$.
On the other hand, we have $X\notin\catdb(R)=\catb_R(R)$.
Note that we have $\supp_R(X)=\{\m_i\}_{\i\in\bbz}$, which is trivially contained in $\spec(R)=\VE(0)$.
So, the failure here is not due to any absence of a support condition.
\end{ex}

\begin{thm}\label{thm151214b}
Let $M$ be an $\fa$-adic semidualizing $R$-complex, and let $X\in\catdb(R)$ be such that $\cosupp_R(X)\subseteq\VE(\fa)$.
Then the following conditions are equivalent.
\begin{enumerate}[\rm(i)]
\item\label{thm151214b1}
$X\in\cata_M(R)$.
\item\label{thm151214b2}
for each multiplicatively closed subset $U\subseteq R$ such that $U^{-1}\fa\neq U^{-1}R$, 
we have $\mathbf{L}\Lambda^{U^{-1}\fa}(U^{-1}X)\in \cata_{U^{-1}M}(U^{-1}R)$.
\item\label{thm151214b3}
For all $\p\in\VE(\fa)$, we have $\mathbf{L}\Lambda^{\fa_\p}(X_\p)\in\cata_{M_\p}(R_\p)$.
\item\label{thm151214b4}
For all $\m\in\VE(\fa)\bigcap\mspec(R)$, we have $\mathbf{L}\Lambda^{\fa_\m}(X_\m)\in\cata_{M_\m}(R_\m)$.
\end{enumerate}
\end{thm}

\begin{proof}
Again, it suffices to prove the implication~\eqref{thm151214b4}$\implies$\eqref{thm151214b1}.
Assume that for all $\m\in\VE(\fa)\bigcap\mspec(R)$, we have $\mathbf{L}\Lambda^{\fa_\m}(X_\m)\in\cata_{M_\m}(R_\m)$.
Foxby Equivalence~\ref{thm121116a}\eqref{thm121116a3} over $R_{\m}$ implies that we have
$\Lotimes[R_\m]{M_\m}{\mathbf{L}\Lambda^{\fa_\m}(X_\m)}\in\catb_{M_\m}(R_\m)$ for all such $\m$. 
By Fact~\ref{fact151222c} we have the following isomorphisms in $\catd(R_\m)$
\begin{equation}\label{eq151214a}
\Lotimes[R_\m]{M_\m}{\mathbf{L}\Lambda^{\fa_\m}(X_\m)}
\simeq \Lotimes[R_\m]{M_\m}{X_\m}
\simeq(\Lotimes MX)_\m
\end{equation}
so we conclude that $(\Lotimes MX)_\m\in\catb_{M_\m}(R_\m)$ for all $\m\in\VE(\fa)\bigcap\mspec(R)$.
Moreover, we have $\supp_R(\Lotimes MX)\subseteq\supp_R(M)\subseteq\VE(\fa)$, by Fact~\ref{cor130528a2}. 

Claim: we have $\Lotimes MX\in\catb_M(R)$.
In light of the previous paragraph, it suffices by Theorem~\ref{thm151214a} to show that 
$\Lotimes MX\in\catdb(R)$. Since the conditions $M,X\in\catdb(R)$ imply that $\Lotimes MX\in\catd_+(R)$,
it suffices to show that $\sup(\Lotimes MX)<\infty$.
To this end, the first two steps in the next sequence are from
Facts~\ref{fact130619b} and~\ref{disc151112a}\eqref{disc151112a5}.
\begin{align*}
\sup(\mathbf{L}\Lambda^{\fa_\m}(X_\m))
&=\sup(\Rhom[R_\m]{\mathbf{R}\Gamma_{\fa_\m}(R_\m)}{X_\m})\\
&\leq\sup(X_\m)-\inf(\mathbf{R}\Gamma_{\fa_\m}(R_\m))\\
&=\sup(X_\m)-\inf(\RG aR_\m)\\
&\leq\sup(X)-\inf(\RG aR)\\
&\leq\sup(X)+\depth_\fa(R)
\end{align*}
The third step is from the standard isomorphism $\mathbf{R}\Gamma_{\fa_\m}(R_\m)\simeq\RG aR_\m$;
one can verify this via the \v Cech complex over $R$.
The fourth step is routine, and the fifth one is Grothendieck's standard non-vanishing result for local cohomology.

For all $\m\in\VE(\fa)\bigcap\mspec(R)$ we have $\mathbf{L}\Lambda^{\fa_\m}(X_\m)\in\cata_{M_\m}(R_\m)$.
Thus, for all such $\m$, the first step in the next sequence is from 
Lemma~\ref{lem151126a}\eqref{lem151126a1}.
\begin{align*}
\sup(\Lotimes[R_\m]{M_\m}{\mathbf{L}\Lambda^{\fa_\m}(X_\m)})
&\leq\sup(\mathbf{L}\Lambda^{\fa_\m}(X_\m))+\sup(M_\m)+n\\
&\leq\sup(X)+\depth_\fa(R)+\sup(M_\m)+n\\
&\leq\sup(X)+\depth_\fa(R)+\sup(M)+n
\end{align*}
The second step is from the preceding paragraph, and the third step is routine.
This explains the last step in the next sequence.
\begin{align*}
\sup(\Lotimes MX)
&=\sup\left\{\sup((\Lotimes MX)_\m)\mid\m\in\Supp_R(\Lotimes MX)\bigcap\mspec(R)\right\} \\
&=\sup\left\{\sup((\Lotimes MX)_\m)\mid\m\in\VE(\fa)\bigcap\mspec(R)\right\} \\
&=\sup\left\{\sup(\Lotimes[R_\m]{M_\m}{\mathbf{L}\Lambda^{\fa_\m}(X_\m)})\mid\m\in\VE(\fa)\bigcap\mspec(R)\right\} \\
&\leq\sup(X)+\depth_\fa(R)+\sup(M)+n
\end{align*}
The first step here is routine, where the ``large support'' of an $R$-complex $Y$ is 
$\Supp_R(Y):=\{\p\in\spec(R)\mid Y_\p\not\simeq 0\}$.
For the second step, use the fact that we have $\supp_R(\Lotimes MX)\subseteq\supp_R(M)\subseteq\VE(\fa)$,
which implies that $\Supp_R(\Lotimes MX)\subseteq\VE(\fa)$, by~\cite[Proposition~3.15(a)]{sather:scc}.
The third step is from the isomorphism~\eqref{eq151214a}.
This establishes the Claim.

To complete the proof, note that 
$\cosupp_R(X)\subseteq\VE(\fa)$ and $\Lotimes MX\in\catb_M(R)$, by the Claim.
So, Foxby Equivalence~\ref{thm121116a}\eqref{thm121116a3} implies that $X\in\cata_M(R)$, as desired.
\end{proof}

The next example shows that one cannot replace $\mathbf{L}\Lambda^{U^{-1}\fa}(U^{-1}X)$ with $U^{-1} X$ in the previous result.

\begin{ex}\label{ex151214b}
Let $k$ be a field and consider the power series ring $R=k[\![Y,Z]\!]$ with $\fa=YR$ and $U=\{1,Z,Z^2,\ldots\}$.
We first show that $U^{-1}R$ is not $U^{-1}\fa$-adically complete.\footnote{This may be well-known, but we include a short proof
for the sake of completeness.} Suppose by way of contradiction that $U^{-1}R$ were $U^{-1}\fa$-adically complete.
Since $U^{-1}\fa=YU^{-1}R$,
it follows that we have 
$$\sum_{i=0}^\infty\frac{1}{Z^i}Y^i\in U^{-1}R=k[\![Y,Z]\!][Z^{-1}].$$
This means that there is a power series $f\in R$ and an integer $m\geq 0$ such that
$$\sum_{i=0}^\infty\frac{1}{Z^i}Y^i=\frac{f}{Z^m}.$$
Clearing the denominator $Z^m$, we find that
$$\sum_{i=0}^\infty\frac{Z^m}{Z^i}Y^i=f\in R=k[\![Y,Z]\!]$$
which is impossible.

Now, Proposition~\ref{prop130612a} shows that $R\in\cata_{\RG aR}(R)$.
On the other hand, we have just  shown that $U^{-1}R$ is not $U^{-1}\fa$-adically complete,
so we have 
$U^{-1}R\notin\cata_{U^{-1}\RG aR}(U^{-1}R)$ by Proposition~\ref{prop140218a}.
\end{ex}

Independent of this example, the derived local homology in Theorem~\ref{thm151214b} is still a bit ugly.
Of course, since localization is a tensor product, and tensor products respect supports (not cosupports),
this is inevitable. On the other hand, Homs respect cosupports, so it makes sense to consider co-localization as well. 
Here, we have only limited results. 

\begin{prop}\label{prop151214a}
Let $M$ be an $\fa$-adic semidualizing $R$-complex, and assume that $\dim(R)<\infty$. 
Let $U\subseteq R$ be multiplicatively closed such that $U^{-1}\fa\neq U^{-1}R$, and let $X\in\catdb(R)$ be given. 
\begin{enumerate}[\rm(a)]
\item\label{prop151214a1}
If $X\in\cata_M(R)$, then $\Rhom{U^{-1}R}X\in\cata_{U^{-1}M}(U^{-1}R)$.
\item\label{prop151214a2}
If $X\in\catb_M(R)$, then $\mathbf{R}\Gamma_{U^{-1}\fa }(\Rhom{U^{-1}R}X)\in\catb_{U^{-1}M}(U^{-1}R)$.
\end{enumerate}
\end{prop}

\begin{proof}
We prove part~\eqref{prop151214a2}; the proof of part~\eqref{prop151214a1} is easier.
Let $Q\colon\catd(U^{-1}R)\to\catd(R)$ be the forgetful functor. 
The fact that $U^{-1}R$ is flat over $R$ implies that we have natural isomorphisms
$Q\circ\mathbf{R}\Gamma_{U^{-1}\fa }\simeq\RGno a\circ Q$ of functors $\catd(U^{-1}R)\to\catd(R)$;
this is easily verified using the \v Cech complex.

Assume that $X\in\catb_M(R)$. The assumption $\dim(R)<\infty$ implies that we have
$\pd_R(U^{-1}R)\leq\dim(R)<\infty$ by~\cite[Theorem II.3.2.6]{raynaud:cpptpm}. Thus, 
the isomorphism from the preceding paragraph conspires with Theorem~\ref{prop151124b} to show that
$$
Q(\mathbf{R}\Gamma_{U^{-1}\fa }(\Rhom{U^{-1}R}X))\simeq\RG a{\Rhom{U^{-1}R}X}\in\catb_M(R).$$
The desired conclusion $\mathbf{R}\Gamma_{U^{-1}\fa }(\Rhom{U^{-1}R}X)\in\catb_{U^{-1}M}(U^{-1}R)$
now follows from Proposition~\ref{prop140218b}\eqref{prop140218b2}.
\end{proof}

\begin{question}\label{q151214a}
Let $M$ be an $\fa$-adic semidualizing $R$-complex, and assume that $\dim(R)<\infty$. 
Let $X\in\catdb(R)$ be given.
\begin{enumerate}[\rm(a)]
\item\label{q151214a1}
If $\cosupp_R(X)\subseteq\VE(\fa)$ and $\Rhom{R_\m}X\in\cata_{M_\m}(R_\m)$ for every $\m\in\VE(\fa)\bigcap\mspec(R)$,
must we have $X\in\cata_M(R)$?
\item\label{q151214a2}
If $\supp_R(X)\subseteq\VE(\fa)$ and 
$\mathbf{R}\Gamma_{\fa R_\m}(\Rhom{R_\m}X)\in\catb_{M_\m}(R_\m)$ for every $\m\in\VE(\fa)\bigcap\mspec(R)$,
must we have $X\in\catb_M(R)$?
\end{enumerate}
\end{question}

\begin{disc}\label{disc151214a}
One may be tempted to attempt to answer Question~\ref{q151214a}\eqref{q151214a1} as in the proof of Theorem~\ref{thm151214a}.
Following this logic, one concludes that
\begin{gather*}
\inf(X)-\dim(R)\leq\inf(\Rhom{R_\m}X)\\
\sup(\Rhom{R_\m}X)\leq\sup(X)
\end{gather*}
for all $\m\in\mspec(R)\bigcap\VE(\fa)$ and that
$$\mathbf{R}\!\operatorname{Hom}\left(\bigoplus_\m R_\m,X\right)\simeq\prod_\m\Rhom{R_\m}X\in\cata_M(R).$$
However, Theorem~\ref{prop151124a} does not allow us to conclude that $X\in\cata_M(R)$.
\end{disc}

\section*{Acknowledgments}
We are grateful to Srikanth Iyengar, 
Liran Shaul,
and Amnon Yekutieli
for helpful comments about this work.

\providecommand{\bysame}{\leavevmode\hbox to3em{\hrulefill}\thinspace}
\providecommand{\MR}{\relax\ifhmode\unskip\space\fi MR }
\providecommand{\MRhref}[2]{%
  \href{http://www.ams.org/mathscinet-getitem?mr=#1}{#2}
}
\providecommand{\href}[2]{#2}

\end{document}